\documentclass[a4paper]{amsart}

\pdfoutput=1

\usepackage[utf8]{inputenc}
\usepackage[T1]{fontenc}
\usepackage{lmodern}
\usepackage{amsthm, amssymb, amsmath, amsfonts, mathrsfs}

\usepackage{microtype}

\usepackage[pagebackref,colorlinks=true,pdfpagemode=none,urlcolor=blue,
linkcolor=blue,citecolor=blue]{hyperref}

\usepackage{amsmath,amsfonts,amssymb,amsthm}
\usepackage{amsthm, amssymb, amsmath, amsfonts, mathrsfs}

\usepackage{mathrsfs}
\usepackage{MnSymbol}
\usepackage{scalerel} 

\usepackage{color}
\usepackage{accents}

\usepackage{bbm}

\newtheorem{theorem}{Theorem}[section]
\newtheorem{lemma}[theorem]{Lemma}

\newtheorem{proposition}[theorem]{Proposition}

\theoremstyle{remark}

\renewenvironment{proof}[1][Proof]{ {\itshape \noindent {#1.}} }{$\Box$
\medskip}

\numberwithin{equation}{section}
\newcommand{\R}{\mathbb{R}}

\newcommand{\Z}{\mathbb{Z}}

\newcommand{\Pb}{\mathbb{P}}

\newcommand{\E}{\mathbb{E}}

\newcommand{\F}{\mathcal{F}}

\newcommand{\G}{\mathcal{G}}

\newcommand{\W}{\mathscr{W}}

\newcommand{\C}{\mathcal{C}}

\newcommand{\cR}{\mathcal{R}}

\newcommand{\V}{\mathcal{V}}
\newcommand{\cV}{\mathscr{V}}

\newcommand{\Q}{\mathcal{Q}}

\newcommand{\cU}{\mathscr{U}}

\newcommand{\eps}{\varepsilon}
\def\les{\lesssim}

\newcommand{\Var}{\mathrm{Var}}

\newcommand{\la}{\langle}
\newcommand{\ra}{\rangle}

\newcommand{\X}{\mathbf{X}}

\newcommand{\Y}{\mathbf{Y}}

\newcommand{\1}{\mathbbm{1}}

\newcommand{\cH}{\mathcal{H}}
\newcommand{\ccH}{\mathscr{H}}

\newcommand{\cP}{\mathcal{P}}
\newcommand{\ccR}{\mathscr{R}}

\begin{document}

\title{Gaussian fluctuations from the 2D KPZ equation}
\author{Yu Gu}

\address[Yu Gu]{Department of Mathematics, Carnegie Mellon University, Pittsburgh, PA 15213, USA}


\maketitle

\begin{abstract}
We prove the two dimensional KPZ equation with a logarithmically tuned nonlinearity and a small coupling constant,   scales to the Edwards-Wilkinson equation with an effective variance.

\bigskip

\noindent \textsc{MSC 2010:} 		35R60, 60H07, 60H15.

\medskip

\noindent \textsc{Keywords:} KPZ equation, Edwards-Wilkinson equation, Feynman-Kac formula.

\end{abstract}
\maketitle

\section{Introduction}
\subsection{Main result}
We are interested in the two dimensional KPZ equation driven by a mollified spacetime white noise and starting from flat initial data:
\begin{equation}\label{e.maineq1}
\partial_t h_\eps=\tfrac12\Delta h_\eps+\tfrac{\beta}{2\sqrt{|\log\eps|}}|\nabla h_\eps|^2 + \dot{W}_\eps(t,x), \    \ h_\eps(0,x)\equiv0,   \quad  x\in\R^2,
\end{equation}
where
\[
\dot{W}_\eps(t,x)=\tfrac{1}{\eps^2} \int_{\R^2} \varphi(\tfrac{x-y}{\eps}) \dot{W}(t,y)dy,
\]
with $\dot{W}$ a spacetime white noise built on the probability space $(\Omega,\F,\Pb)$ and $0\leq \varphi\in \C_c^\infty(\R^2)$. The covariance function of $\dot{W}_\eps$ is formally written as
\begin{equation}\label{e.covfunction}
\E[\dot{W}_\eps(t,x)\dot{W}_\eps(s,y)]=\delta(t-s)\tfrac{1}{\eps^2}R(\tfrac{x-y}{\eps}), \quad \mbox{ with } \quad R(x)=\int_{\R^2}\varphi(x+y)\varphi(y)dy.
\end{equation}
Without loss of generality, we assume $\varphi(x)=0$ for $|x|\geq\tfrac12$ and  $\int_{\R^2} \varphi(x)dx=1$. The following is our main result:
\begin{theorem}\label{t.mainth}
There exists $\beta_0$ depending on $\varphi$ such that if $\beta<\beta_0\leq\sqrt{2\pi}$,  then for any $t>0$ and test function $g\in\C_c^\infty(\R^2)$, we have
\begin{equation}\label{e.mainre}
\int_{\R^2} \left(h_\eps(t,x)-\E[h_\eps(t,x)]\right)g(x)dx
\Rightarrow \int_{\R^2} \ccH(t,x) g(x)dx
\end{equation}
in distribution as $\eps\to0$, where $\ccH$ solves the Edwards-Wilkinson equation 
\[
\partial_t \ccH=\tfrac12\Delta \ccH+ \nu_{\mathrm{eff}}  \dot{W}(t,x), \   \ \ccH(0,x)\equiv 0,
\]
with the effective variance
\begin{equation}\label{e.effvar}
\nu_{\mathrm{eff}}^2= \tfrac{2\pi }{2\pi-\beta^2}.
\end{equation}
\end{theorem}

There is a lot of activities on the study of singular SPDEs over the past decade. We refer to the reviews \cite{corwin2012kardar,corwinshen,quastel2015one} and the references therein. For the KPZ equation, progresses in $d\geq3$ can be found in \cite{dunlap2018kpz,magnen2017diffusive}, where results similar to Theorem~\ref{t.mainth} were proved. In two dimensions, the tightness of $\{h_\eps\}_{\eps\in(0,1)}$, as a sequence of random distributions, was proved in the recent work of Chatterjee-Dunlap \cite{chatterjee2018constructing}. To prove Theorem~\ref{t.mainth}, we implement the same strategy laid out in \cite{dunlap2018kpz}. 

The convergence in \eqref{e.mainre} is expected to hold for all $\beta\in(0,\sqrt{2\pi})$, and our proof seems to only work for $\beta$ small enough. Near the completion of this paper, we learnt the very recent work of Caravenna-Sun-Zygouras \cite{CRS18} which proved Theorem~\ref{t.mainth} for all $\beta\in(0,\sqrt{2\pi})$, using a different method. While their result is more general and covers the entire ``subcritical'' regime, the proof presented here seems to be simpler and offers another perspective. We discuss the approaches of \cite{CRS18,chatterjee2018constructing} in Section~\ref{s.comparison}.

At the critical value $\beta=\sqrt{2\pi}$, the early work of Bertini-Cancrini \cite{bertini1998two} identified the limiting covariance function of the corresponding stochastic heat equation. While the limiting distribution remains an open question, we refer to the work of \cite{CSZ18,feng2016rescaled} in this direction.

If we write the nonlinear term in \eqref{e.maineq1} as $|\nabla h_\eps|^2=\nabla h_\eps \cdot \Lambda\nabla h_\eps$, our case corresponds to $\Lambda$ being the $2\times 2$ identity matrix. The so-called \emph{anisotropic} class refers to the case of $\mathrm{det}[\Lambda]<0$. It is a very interesting question to study the asymptotics of the anisotropic version of  \eqref{e.maineq1}, where the Hopf-Cole transformation to the stochastic heat equation is unavailable and all existing methods break down. Some recent work on the interacting particle systems belonging to this class can be found in \cite{corwin1,corwin2,chen,ton}.

While we always view $\{h_\eps\}_{\eps>0}$ as a family of random distributions in $d\geq2$ and study the asymptotics of $\int h_\eps g$ with test function $g$, there are also recent studies on the pointwise fluctuations of $h_\eps$ (or $e^{h_\eps}$), see e.g. \cite{caravenna2015universality,CometsMukh,CometsMukh1,Cosco,mukherjee2016weak}.

\subsection{Connection to the stochastic heat equation and heuristics}\label{s.heuristics}

Through a Hopf-Cole transformation, the $h_\eps$ defined in \eqref{e.maineq1} is related to the solution of the heat equation with a weak random potential
\begin{equation}\label{e.maineq}
\partial_t u=\tfrac12\Delta u+\beta_\eps V(t,x)u, \    \  u(0,x)\equiv1,  \   \  x\in\R^2,
\end{equation}
with
\begin{equation}\label{e.defbetaeps}
\beta_\eps=\tfrac{\beta}{\sqrt{|\log\eps|}},
\end{equation}
 and 
\[
V(t,x)=\int_{\R^2} \varphi(x-y)\dot{W}(t,y)dy.
\] 
Here, the product $V(t,x)u$ in \eqref{e.maineq} 
is interpreted in the It\^o sense. Consider $u_\eps(t,x)=u(\tfrac{t}{\eps^2},\tfrac{x}{\eps})$,
which solves 
\[
\partial_t u_\eps=\tfrac12\Delta u_\eps+\tfrac{\beta_\eps}{\eps^2} V(\tfrac{t}{\eps^2},\tfrac{x}{\eps}) u_\eps.
\]
By the scaling property of the spacetime white noise and the fact that $d=2$, we have 
\begin{equation}\label{e.lawV}
\tfrac{1}{\eps^2}V(\tfrac{t}{\eps^2},\tfrac{x}{\eps}) \stackrel{\text{law}}{=}\dot{W}_\eps(t,x). 
\end{equation}
Applying It\^o's formula yields
\[
\beta_{\eps}^{-1}(\log u_\eps -\E[\log u_\eps])\stackrel{\text{law}}{=}h_\eps-\E[h_\eps].
\] 
From now on, we will study $\log u_\eps$ rather than $h_\eps$. 

Our proof of Theorem~\ref{t.mainth} implies a similar result of $u_\eps$: for $\beta<\beta_0$,
\begin{equation}\label{e.conshe}
 \beta_\eps^{-1}\int_{\R^2}(u_\eps(t,x)-1)g(x)dx\Rightarrow \int_{\R^2}\ccH(t,x)g(x)dx, \quad \mbox{ in distribution. } 
\end{equation}
This was previously proved in \cite[Theorem 2.17]{caravenna2015universality} for all $\beta\in(0,\sqrt{2\pi})$. Let us explain the mechanism behind the convergence of \eqref{e.conshe} for the stochastic heat equation (SHE) and how it relates to the convergence of the KPZ in \eqref{e.mainre}. 

First, the variance of the l.h.s. of \eqref{e.conshe} is
\begin{equation}\label{e.varshe}
\Var[\beta_\eps^{-1}\int_{\R^2}u_\eps(t,x)g(x)dx]=\beta^{-2}|\log\eps| \int_{\R^4}\mathrm{Cov}[u(\tfrac{t}{\eps^2},\tfrac{x}{\eps}),u(\tfrac{t}{\eps^2},\tfrac{y}{\eps})]g(x)g(y)dxdy.
\end{equation}
The covariance is written explicitly by the Feynman-Kac formula:
\begin{equation}\label{e.covre}
\begin{aligned}
&u(t,x)=\E_B[ e^{\beta_\eps\int_0^{t} V(t-s,x+B_s)ds-\frac12\beta_\eps^2 R(0)t}],\\
&\mathrm{Cov}[u(\tfrac{t}{\eps^2},\tfrac{x}{\eps}),u(\tfrac{t}{\eps^2},\tfrac{y}{\eps})]=\E_B\big[e^{\beta^2|\log\eps|^{-1}\int_0^{t/\eps^2} R(\tfrac{x-y}{\eps}+B^1_s-B^2_s)ds}\big]-1=F(\tfrac{t}{\eps^2},\tfrac{x-y}{\eps})-1,
\end{aligned}
\end{equation}
where $B^1,B^2$ are independent Brownian motions starting from the origin and $\E_B$ denotes the expectation with respect to the Brownian motions. Note that in the Feynman-Kac representation of $u$, we have the factor $\int_0^t V(t-s,x+B_s)ds$. While the Brownian motion $B_s$ starts from the origin, it is sometimes more convenient for us to write the integral as $\int_0^t V(s,x+B_{t-s})ds$ and view $\{x+B_{t-s}\}_{s\geq0}$ as a Brownian motion starting at $(t,x)$ and running backwards in time. The function $F$ in \eqref{e.covre} solves the deterministic PDE 
\[
\partial_t F=\Delta F+\beta^2|\log\eps|^{-1} R(x)F, \quad F(0,x)\equiv1.
\]
Similar to $u$, we have omitted the dependence of $F$ on $\eps$. The above equation can be written in the mild formulation as
\[
F(\tfrac{t}{\eps^2},\tfrac{x-y}{\eps})=1+\beta^2|\log\eps|^{-1}\int_0^{t/\eps^2}\int_{\R^2} G_{2(t/\eps^2-\ell)}(\tfrac{x-y}{\eps}-w)R(w)F(\ell,w)dwd\ell,
\]
where $G_t(x)=(2\pi t)^{-1}\exp(-|x|^2/2t)$ is the standard heat kernel. After a change of variable $\ell\mapsto \ell/\eps^2$, we have 
\begin{equation}\label{e.covare1}
\beta^{-2}|\log \eps|\,[F(\tfrac{t}{\eps^2},\tfrac{x-y}{\eps})-1]= \int_0^t \int_{\R^2}G_{2(t-\ell)}(x-y-\eps w)R(w) F(\tfrac{\ell}{\eps^2},w)dwd\ell.
\end{equation}
By the Feynman-Kac representation of $F$ in \eqref{e.covre}, we know that $F(\ell/\eps^2,w)$ measures the intersection time of the two Brownian motions during $[0,\ell/\eps^2]$. By a classical result of Kallianpur-Robbins \cite[Theorem 1]{KR},  for $\beta>0, \ell>0$ and $w\in\R^2$, the following convergence in distribution holds:
\begin{equation}\label{e.wkconKR}
|\log \eps|^{-1}\int_0^{\ell/\eps^2} R(w+B^1_s-B^2_s)ds\stackrel{\text{law}}{=}(2|\log\eps|)^{-1}\int_0^{2\ell/\eps^2} R(w+B_s)ds\Rightarrow \tfrac{1}{2\pi}\mathrm{Exp}(1),
\end{equation}
where we used the fact that $\int R=1$. Together with the uniform integrability we will establish later, this implies 
\begin{equation}\label{e.1271}
F(\tfrac{\ell}{\eps^2},w)=\E_B[ e^{\beta^2 |\log \eps|^{-1} \int_0^{\ell/\eps^2} R(w+B^1_s-B^2_s)ds}]\to \tfrac{2\pi}{2\pi-\beta^2 }=\nu_{\mathrm{eff}}^2
\end{equation}
for small $\beta$, as $\eps\to0$. Combining \eqref{e.covare1} and \eqref{e.1271}, the variance in \eqref{e.varshe} converges:
\begin{equation}\label{e.convar1210}
\begin{aligned}
\Var[\beta_\eps^{-1}\int_{\R^2}u_\eps(t,x)g(x)dx]\to& \,\nu_{\mathrm{eff}}^2\int_0^t\int_{\R^6} G_{2(t-\ell)}(x-y)R(w)g(x)g(y)dxdydwd\ell\\
&=\Var[\int_{\R^2}\ccH(t,x)g(x)dx].
\end{aligned}
\end{equation}
In the last step we used the fact that $\int R=(\int \varphi)^2=1$. While the effective variance $\nu_{\mathrm{eff}}^2$ only depends on the mollifier through the integral $\int \varphi$ in our setting, it is not the case at the critical value $\beta=\sqrt{2\pi}$ in $d=2$ or in higher dimensions $d\geq 3$, see \cite{bertini1998two,dunlap2018kpz,GRZ17}, which is very different from the subcritical setting. The above calculation and the convergence in \eqref{e.1271} interprets the effective variance $\nu_{\mathrm{eff}}^2$ in terms of the intersection of two Brownian paths.

Now we explain the origin of the Gaussianity. It is important to note that the main contribution to the integral in \eqref{e.wkconKR} comes from $s\in[0,K_\eps]$ provided that $|\log (\eps^2 K_\eps)|\ll |\log \eps|$. Actually, the heat kernel in $d=2$ satisfies that $G_t(x)\sim t^{-1}$ for $x$ near the origin, so we have
\begin{equation}\label{e.keps}
\begin{aligned}
|\log \eps|^{-1}\left(\int_0^{\ell/\eps^2}-\int_0^{K_\eps}\right) &\E_B[R(w+B^1_s-B^2_s)]ds\\
&=|\log \eps|^{-1} \int_{K_\eps}^{\ell/\eps^2} \E_B[R(w+B^1_s-B^2_s)]ds\\
&\sim  |\log \eps|^{-1} |\log (\eps^2 K_\eps)|\to0.
\end{aligned}
\end{equation}
For example, we can pick $K_\eps= \tfrac{1}{\eps^2|\log \eps|}=o(\eps^{-2})$ and replace $F(\ell/\eps^2,w)$ in \eqref{e.covare1} by $F(K_\eps,w)$ without changing the asymptotic covariance:
\begin{equation}\label{e.covare2}
\begin{aligned}
\beta^{-2}|\log \eps|\,[F(\tfrac{t}{\eps^2},\tfrac{x-y}{\eps})-1]\approx&  \int_0^t \int_{\R^2}G_{2(t-\ell)}(x-y-\eps w)R(w) F(K_\eps,w)dwd\ell.
\end{aligned}
\end{equation}
Recall that 
\[
\Var[\beta_\eps^{-1}\int_{\R^2}u_\eps(t,x)g(x)dx]=\int_{\R^4} \beta^{-2}|\log \eps|\,[F(\tfrac{t}{\eps^2},\tfrac{x-y}{\eps})-1]g(x)g(y)dxdy.
\]
The r.h.s. of \eqref{e.covare2} indicates the main contribution to the variance of our interested quantity, from the perspective of Brownian paths intersections. In microscopic variables, we have two Brownian paths, starting from $\tfrac{x}{\eps}$ and $\tfrac{y}{\eps}$ respectively and running backwards in time. After first ``meeting'' each other at the time $(t-\ell)/\eps^2$ for some $\ell\in(0,t)$, the two paths spend $K_\eps=o(\eps^{-2})$ amount of time ``intersecting'' before splitting again. As a result, the random  environment  involved in this ``intersection'' only consists of $\dot{W}(s,\cdot)$ with $s\in[\ell/\eps^2-o(\eps^{-2}),\ell/\eps^2]$, which induces a temporal decorrelation for  different $\ell_1\neq \ell_2\in(0,t)$ and creates the Gaussianity. Together with the variance convergence in \eqref{e.convar1210}, we have the Edwards-Wilkinson limit in \eqref{e.conshe}. The results in \cite{GRZ17} for $d\geq3$ is based on the above heuristics.

For the KPZ equation, the Gaussianity comes from a similar temporal decorrelation as discussed above (we will prove it by a different method though). The convergence of the variance 
\[
\begin{aligned}
\Var[\int_{\R^2}h_\eps(t,x)g(x)dx]&=\Var[\beta_\eps^{-1}\int_{\R^2} \log u_\eps(t,x)g(x)dx]\\
&\to\Var[\int_{\R^2}\ccH(t,x)g(x)dx]
\end{aligned}
\]
is however more involved. While we do not have a Feynman-Kac representation for $\mathrm{Cov}[\log u_\eps(t,x),\log u_\eps(t,y)]$ as \eqref{e.covre}, an application of the Clark-Ocone formula will help us express the covariance in terms of an integral of  
\begin{equation}\label{e.dlogu}
  \E[D\log u_\eps(t,z) |\F_r]=\E[u_\eps^{-1}(t,z) Du_\eps(t,z) |\F_r],  \quad  z=x,y, \quad r\leq \tfrac{t}{\eps^2}.
\end{equation}
Here $D$ is the Malliavin derivative with respect to the random noise and $\F_r$ is the filtration generated by $\{\dot{W}(s,\cdot), s\leq r\}$. The key difficulty in analyzing \eqref{e.dlogu} is to deal with the factor $u_\eps^{-1}$ and to evaluate the conditional expectation given $\F_r$. By the same discussion for \eqref{e.keps}, the random variable $u_\eps(t,\cdot)$ mainly depends on the noise $\dot{W}(s,\cdot)$ for $s\in[t/\eps^2-o(\eps^{-2}),t/\eps^2]$, so we could replace the factor $u_\eps^{-1}(t,\cdot)$ in \eqref{e.dlogu} with a small error by something that is independent of $\F_r$ for those  $r<\tfrac{t}{\eps^2}-o(\eps^{-2})$. The rest of the discussion is similar to the SHE case.


\subsection{Notation}
\label{sec-notation}
We use the following notation and conventions.

\begin{enumerate}
\item We use $a \les b$ for $a\leq Cb$ for some constant $C$ independent of $\eps$.

\item We use $(p,q)$ to denote the H\"older exponents $\tfrac{1}{p}+\tfrac{1}{q}=1$, and always choose $p\gg1$.

\item $G_t(x)=(2\pi t)^{-1}\exp(-|x|^2/2t)$ denotes the standard heat kernel.

\item
We let 
$H$ denote the Hilbert space $L^2(\R^{2+1})$, with norm
$\|\cdot\|_H$ and inner product $\la\cdot,\cdot\ra_H$.


\item $\{B^j_t: t\geq0, j=1,\ldots\}$ is a family of standard independent $2-$dimensional Brownian motions built on another probability space $(\Sigma,\mathcal{A},\Pb_B)$.  We will use $\E_B,\Pb_B$ when taking the expectation and the probability with respect to $B$.

\item We use $\|\cdot\|_p$ to denote the $L^p$ norm of the product probability space $(\Omega\times \Sigma,\F\otimes \mathcal{A},\Pb\times \Pb_B)$ for $p\geq1$.

\item We use $d_{\mathrm{TV}}(\cdot,\cdot)$ to denote the total variation 
distance between two distributions.

\item We let $\|\cdot \|_{\mathrm{op}}$ denotes the operator norm.

\item We use $[0,t]_<^n$ to denote the $n-$dimensional simplex $\{0\leq t_1<\ldots<t_n\leq t\}$.
\item We use $\hat{f}(\xi)=\int_{\R^d}f(x)e^{-i\xi \cdot x}dx$ to denote the Fourier transform of $f$.
\end{enumerate}

\subsection*{Acknowledgments} We would like to thank Li-Cheng Tsai for his initial involvement in this project and multiple inspiring discussions. We thank Nikolaos Zygouras for some helpful discussions, and two anonymous referees for a very careful reading of the manuscript and many helpful suggestions to improve the presentation. The research is supported by NSF grant DMS-1613301/1807748/1907928 and the Center for Nonlinear Analysis of CMU.

\section{Sketch of the proof}\label{s.sketch}

The main result \eqref{e.mainre} is equivalent with the convergence in distribution of
\begin{equation}\label{e.mainre1}
\beta_{\eps}^{-1}\int_{\R^2} (\log u_\eps(t,x)-\E[\log u_\eps(t,x)])g(x)dx \Rightarrow \int_{\R^2} \ccH(t,x)g(x)dx.
\end{equation}
We rely on the Feynman-Kac representation of the solution to \eqref{e.maineq}:
\begin{equation}\label{e.fkre}
u(t,x)=\E_B\left[ e^{\beta_\eps \int_0^t V(t-s,x+B_s)ds -\frac12\beta_\eps^2R(0)t}\right],
\end{equation}
which has the same distribution, if viewed as a random field in $x$ with $t$ fixed, as 
\[
Z(t,x)=\E_B[M(t,x)], \quad \mbox{ with } \quad M(t,x)=\exp\left(\beta_\eps\int_0^t V(s,x+B_s)ds-\tfrac12\beta_\eps^2R(0)t\right).
\]
We keep in mind that $M,Z$ depend on $\eps$ through the small factor $\beta_\eps$ defined in \eqref{e.defbetaeps} but omit its dependence to simplify the notation. For fixed $B$ and $x$, $M(\cdot,x)$ is a martingale. 
Defining
\[
Z_\eps(t,x)=Z(\tfrac{t}{\eps^2},\tfrac{x}{\eps}), \     \ M_\eps(t,x)=M(\tfrac{t}{\eps^2},\tfrac{x}{\eps}),
\]
and
\[
X_\eps(t)=\int_{\R^2} \log Z_\eps(t,x)g(x)dx.
\]
The convergence in \eqref{e.mainre1} is equivalent to 
\[
\beta_\eps^{-1}(X_\eps(t)-\E[X_\eps(t)])\Rightarrow \int_{\R^2}\ccH(t,x) g(x)dx.
\]
Throughout the paper, we fix the variable $t>0$ and sometimes omit its dependence. Define
\begin{equation}\label{e.defsigma}
\begin{aligned}
\sigma_t^2=\Var[\int \ccH(t,\cdot)g(\cdot)]= \nu_{\mathrm{eff}}^2 \int_0^t \int_{\R^{4}} g(x_1)g(x_2) G_{2s}(x_1-x_2)dx_1dx_2ds,
\end{aligned}
\end{equation}
where we recall that
$G_t(x)$ is the standard heat kernel. The proof of Theorem~\ref{t.mainth} consists of two steps:

\begin{proposition}\label{p.convar}
As $\eps\to0$, $\beta_\eps^{-2}\Var[X_\eps(t)]\to \sigma_t^2$.
\end{proposition}

\begin{proposition}\label{p.gauss}
As $\eps\to0$,
\[
\frac{X_\eps(t)-\E[X_\eps(t)]}{\sqrt{\Var[X_\eps(t)]}}\Rightarrow N(0,1).
\]
\end{proposition}

\subsection{Negative moments}

Throughout the paper, we rely on the existence of negative moments of $Z_\eps(t,x)$ for small $\beta$, which essentially comes from \cite{hu2018asymptotics} and was also presented in \cite[Equation (5.13)]{CRS18}.

\begin{proposition}\label{p.negative}
There exits $\beta_0>0$ such that if $\beta<\beta_0$, 
\[
\sup_{t\in[0,T]}\sup_{\eps\in(0,1)}\E[Z_\eps(t,x)^{-n}] \leq C_{\beta,n,T}.
\]
\end{proposition}

The proof is presented in Appendix~\ref{s.nemm}.

\subsection{The Clark-Ocone representation}

For each realization of the Brownian motion $B$, we can write 
\[
\begin{aligned}
\int_0^{t/\eps^2} V(s,\tfrac{x}{\eps}+B_s)ds=&\int_0^{t/\eps^2} \left(\int_{\R^2} \varphi(\tfrac{x}{\eps}+B_s-y)\dot{W}(s,y)dy\right)ds\\
=&\int_{\R^{3}} \Phi_{t,x,B}^\eps(s,y)dW(s,y),
\end{aligned}
\]
with 
\[
\Phi_{t,x,B}^\eps(s,y)=\1_{[0,t/\eps^2]}(s) \varphi(\tfrac{x}{\eps}+B_s-y).
\]
Therefore, 
\[
\begin{aligned}
D_{s,y} Z_\eps(t,x)=D_{s,y} \E_B[M_\eps(t,x)]
= \beta_\eps \E_B\left[ M_\eps(t,x)\Phi_{t,x,B}^\eps(s,y)\right],
\end{aligned}
\]
where $D_{s,y}$ denotes the Malliavin derivative operator with respect to $\dot{W}$. By \cite[Lemma A.1]{dunlap2018kpz}, we have 
\[
D_{s,y}\log Z_\eps(t,x)= \frac{D_{s,y} Z_\eps(t,x)}{Z_\eps(t,x)},
\]
and the Clark-Ocone formula says
\begin{eqnarray}
\nonumber
X_\eps-\E[X_\eps]&=&\int_{\R^{3}} \E[D_{s,y} X_\eps |\F_s] dW(s,y)\\
  \label{eq-star1}
&=&\int_{\R^{3}} \E\big[ \int_{\R^2} \frac{D_{s,y}Z_\eps(t,x)}{Z_\eps(t,x)}g(x)dx \big| \F_s\big]dW(s,y)\\
&=&\beta_\eps\int_0^{t/\eps^2}\int_{\R^2}\left(\int_{\R^2} g(x)\E\big[ \frac{\E_B[ M_\eps(t,x)\Phi^\eps_{t,x,B}(s,y)]}{Z_\eps(t,x)} \big| \F_s\big]dx \right) dW(s,y).
\nonumber
\end{eqnarray}
Here $\F_s$ is the filtration generated by $\dot{W}(\ell, \cdot)$ up to $\ell\leq s$.

For \begin{equation}\label{e.defKeps}
K_\eps=\tfrac{1}{\eps^2|\log\eps|^\alpha}
\end{equation} with some $\alpha>0$ to be determined, we decompose the  stochastic integral in \eqref{eq-star1}
into three parts:
\[
\beta_\eps^{-1}(X_\eps-\E[X_\eps])=I_{1,\eps}+I_{2,\eps}+I_{3,\eps},
\]
with 
\begin{equation}\label{e.defI1}
I_{1,\eps}=\int_0^{K_\eps}\int_{\R^2}\left(\int_{\R^2} g(x)\E\big[ \frac{\E_B[ M_\eps(t,x)\Phi^\eps_{t,x,B}(s,y)]}{Z_\eps(t,x)} \big| \F_s\big]dx \right) dW(s,y),
\end{equation}
\begin{equation}\label{e.defI2}
\begin{aligned}
I_{2,\eps}=\int_{K_\eps}^{t/\eps^2} \int_{\R^2}\left(\int_{\R^2} g(x)\E\big[ \frac{\E_B[ M_\eps(t,x)\Phi^\eps_{t,x,B}(s,y)]}{Z(K_\eps,x/\eps)}\left(\frac{Z(K_\eps,x/\eps)}{Z(t/\eps^2,x/\eps)}-1\right) \big| \F_s\big]dx \right) dW(s,y),
\end{aligned}
\end{equation}
and
\begin{equation}
  \label{e.defI3}
I_{3,\eps}=\int_{K_\eps}^{t/\eps^2} \int_{\R^2}\left(\int_{\R^2} g(x)\E\big[ \frac{\E_B[ M_\eps(t,x)\Phi^\eps_{t,x,B}(s,y)]}{Z(K_\eps,x/\eps)}\big| \F_s\big]dx \right) dW(s,y).
\end{equation}
Since $1\ll K_\eps\ll \eps^{-2}$ and we expect that $Z(K_\eps,x/\eps)$ is close to $Z(t/\eps^2,x/\eps)$, 
 the contribution from $I_{1,\eps},I_{2,\eps}$ is small compared to that from $I_{3,\eps}$. For $I_{3,\eps}$, the integration is in $s\geq K_\eps$, so the random variable $Z(K_\eps,x/\eps)$ is $\F_s-$measurable, and 
\begin{equation}\label{e.defI3new}
I_{3,\eps}=\int_{K_\eps}^{t/\eps^2} \int_{\R^2}\left(\int_{\R^2} \frac{g(x)}{Z(K_\eps,x/\eps)}\E\big[ \E_B[ M_\eps(t,x)\Phi^\eps_{t,x,B}(s,y)]\big| \F_s\big]dx \right) dW(s,y).
\end{equation}
Note that the procedure we took here is slightly different from the heuristics provided in Section~\ref{s.heuristics} due to the time reversal and the fact that we considered $Z(t,x)$ rather than $u(t,x)$. Mathematically they are equivalent.

\subsection{The second order Poincar\'e inequality}

To simplify the notation, we define 
\[
Y_\eps=\tfrac{X_\eps-\E[X_\eps]}{\sqrt{\Var[X_\eps]}}.
\] To show that 
$Y_\eps\Rightarrow N(0,1)$, we apply the second order Poincar\'e inequality, which was originally proved in the discrete setting in \cite{chatterjee2009fluctuations} and generalized to the continuous setting in \cite{nourdin2009second}. Since $\E[Y_\eps]=0$ and $\Var[Y_\eps]=1$, 
with $\zeta$ a standard centered Gaussian random variable, by \cite[Theorem 1.1]{nourdin2009second}, we have
\begin{equation}\label{e.2ndpoincare}
d_{\mathrm{TV}}(Y_\eps,\zeta) \les \E[\|DY_\eps\|_H^4]^{1/4} \E[\|D^2Y_\eps\|_{\mathrm{op}}^4]^{1/4},
\end{equation}
where we recall $H=L^2(\R^{2+1})$ and $\|D^2Y_\eps\|_{\mathrm{op}}$ denotes the operator norm of the mapping $H\otimes H\ni D^2Y_\eps: H \to H$ defined as $D^2 Y_\eps h:=\la h,D^2 Y_\eps\ra_H$, i.e., for any $h\in H$, we have $[D^2 Y_\eps h](s,y)=\int_{\R^{2+1}} h(s',y')D_{s',y'}D_{s,y}Y_\eps ds'dy'$.

Since
\[
DY_\eps=\tfrac{DX_\eps}{\sqrt{\Var[X_\eps]}}, \    \ D^2 Y_\eps=\tfrac{ D^2X_\eps}{\sqrt{\Var[X_\eps]}},
\]
and $\Var[X_\eps]\sim |\log\eps|^{-1}$ from Proposition~\ref{p.convar}, 
 to show $d_{\mathrm{TV}}(Y_\eps,\zeta)\to0$ using \eqref{e.2ndpoincare}, we only need to prove
\begin{equation}\label{e.sept41}
\E[\|DX_\eps\|_H^4]^{1/4} \E[\|D^2X_\eps\|_{\mathrm{op}}^4]^{1/4}=o(|\log\eps|^{-1}),  \mbox{ as } \eps\to0.
\end{equation}

Another possible way to prove the Gaussianity is to utilize the fast temporal mixing, as explained heuristically in Section~\ref{s.heuristics} and implemented in $d\geq3$ in \cite{GRZ17}.

\subsection{Discussions and remarks}
\label{s.comparison}

\subsubsection*{A comparison with \cite{CRS18} and \cite{chatterjee2018constructing}}
The basic ideas behind our approach and the one in \cite{CRS18} are similar, and the key is to modify the partition function so that $\log Z(\tfrac{t}{\eps^2},\tfrac{x}{\eps})$ can be ``linearized'' in some sense. As we will prove later, $I_{3,\eps}$ is the main contribution to the random fluctuations, which essentially corresponds to the partition function of a directed polymer $\{B_s\}_{s\geq0}$ that interacts only with the random environment $\dot{W}(s,\cdot)$ in $s\geq K_\eps$. The initial layer in $s< K_\eps$ only determines the starting point $B_{K_\eps}$ for this interaction. By our choice of $K_\eps=o(\eps^{-2})$, it is easy to show that in the weak disorder regime ($\beta$ small), the polymer behaves like the Brownian motion so $\eps B_{K_\eps} \to0$ as $\eps\to0$, which indicates that the initial layer plays no role in the limit. Given this heuristics, if we ignore the factor $Z(K_\eps,x/\eps)^{-1}$ in \eqref{e.defI3new}, then $I_{3,\eps}$ becomes 
\[
\begin{aligned}
I_{3,\eps}\mapsto &\int_{K_\eps}^{t/\eps^2} \int_{\R^2}\left(\int_{\R^2} g(x)\E\big[ \E_B[ M_\eps(t,x)\Phi^\eps_{t,x,B}(s,y)]\big| \F_s\big]dx \right) dW(s,y)\\
&= \int_{\R^2} g(x) \left(\int_{K_\eps}^{t/\eps^2} \int_{\R^2} \E[ D_{s,y}Z(\tfrac{t}{\eps^2},\tfrac{x}{\eps}) |\F_s]dW(s,y)\right)dx.
\end{aligned}
\]
The last expression precisely describes the fluctuation of the partition function that only involves the environment in $s\geq K_\eps$. A similar heuristics was given in \cite[Section 2.1]{CRS18}, and the $Z^{B\geq}_{N,\beta_N}(x)-1$ defined in \cite[Equation (2.11)]{CRS18} corresponds to the above expression. The Clark-Ocone formula seems to be particularly handy for this ``linearization''. We also note a naive Taylor expansion does not necessarily work for $\mathfrak{f}(Z(\tfrac{t}{\eps^2},\tfrac{x}{\eps}))$ with arbitrary smooth $\mathfrak{f}$, as shown in \cite[Theorem 1.2]{dunlap2018kpz}.

Another difference between the two approaches is the proof of the Gaussianity. After the ``linearization'' in \cite[Proposition 2.3]{CRS18}, the convergence to the Edwards-Wilkinson limit follows from the convergence of SHE proved in \cite[Theorem 2.17]{caravenna2015universality}, which was based on a polynomial chaos expansion and the fourth moment theorem \cite{nourdin2010,nualart05}. In our case, we directly apply the second order Poincar\'e inequality to the KPZ equation, which simplifies some analysis. On the other hand, the fourth moment theorem covers more general distributions of the random environment and the convergence of a discrete directed polymer model to the Edwards-Wilkinson limit was proved in \cite[Theorem 1.6]{CRS18}, while we only deal with the continuous Gaussian environment in our setting. 

The approach in \cite{chatterjee2018constructing} relies on the Feynman-Kac formula and the concentration inequality to control the intersection time of two polymer paths. While a naive application of the Gaussian-Poincar\'e inequality fails for a similar reason as our Lemma~\ref{l.1stde} does not allow $\delta\to0$, the authors have designed a clever recursive scheme that is similar to perturbative renormalization, using which they obtained the desired estimates to prove the tightness.

\subsubsection*{The assumption of $\beta\ll1$} Throughout the paper, we assume $\beta<\beta_0$ for some $\beta_0\ll1$, which is used to  control expectations of the form $\E_B[ e^{\beta^2 f_\eps(B)}g_\eps(B)]$, where $f_\eps,g_\eps$ are Brownian functionals and $f_\eps$ measures the intersection time of multiple pairs of independent Brownian motions. Take $f_\eps(B)=|\log\eps|^{-1}\int_0^{t/\eps^2} R(B^1_s-B^2_s)ds$ for example. It is easy to  show that $\E_B[ e^{\beta^2 f_\eps(B)}]\les1$ for all $\beta<\sqrt{2\pi}$, and we can view 
\[
\E_B[ e^{\beta^2 f_\eps(B)}g_\eps(B)]=\E_B[e^{\beta^2 f_\eps(B)}]\frac{\E_B[ e^{\beta^2 f_\eps(B)}g_\eps(B)]}{\E_B[e^{\beta^2 f_\eps(B)}]}
\] as the average of $g_\eps(B)$ under the annealed polymer measure. In the weak disorder regime, the polymer behaves like the Brownian motion, so ideally we would like to control $\E_B[ e^{\beta^2 f_\eps(B)}g_\eps(B)]$ in terms of $\|g_\eps(B)\|_1$. As it is hard to achieve this, we sacrifice to use the H\"older's inequality
\[
\E_B[ e^{\beta^2 f_\eps(B)}g_\eps(B)]\les \E_B[ e^{p\beta^2f_\eps(B)}]^{1/p} \|g_\eps(B)\|_q
\]
 for the H\"older exponent $p,q>1$. To make the error small enough for our purpose, we need $q$ to be close to $1$ so that $\|g_\eps(B)\|_q$ is close to $\|g_\eps(B)\|_1$. As a result, the dual exponent $p$ is large which put a more restrictive condition on $\beta$ through the constraint of $p\beta^2<2\pi$. From this perspective, the method of using the chaos expansion and the hypercontractivity \cite[Theorem B.1]{CRS18} is preferable in obtaining sharp estimates, compared to analyzing the Feynman-Kac representation and  Brownian functionals.
 
\subsubsection*{Asymptotics of $\E[h_\eps]$} While we are only interested in the random fluctuations of $h_\eps$ here, it is natural to ask whether  the average height $\E[h_\eps]$ can be expressed more explicitly in terms of $\eps$. 
 Recall that $\beta_\eps=\beta|\log\eps|^{-1/2}$ and $u_\eps$ solves the equation interpreted in the It\^o's sense. By It\^o's formula, we know that $\beta_\eps^{-1}\log u_\eps\stackrel{\text{law}}{=} \tilde{h}_\eps$ with 
 \[
 \partial_t \tilde{h}_\eps=\tfrac12\Delta \tilde{h}_\eps+\tfrac12\beta_\eps|\nabla \tilde{h}_\eps|^2+\dot{W}_\eps-\tfrac{\beta_\eps R(0)}{2\eps^2}.
 \]
 By the negative and positive moments bound on $u_\eps$, we have
 \[
 \E[\tilde{h}_\eps]=\E[\beta_\eps^{-1}\log u_\eps]=O(|\log \eps|^{\frac12}).
 \]
 Compare to the equation \eqref{e.maineq1} satisfied by $h_\eps$, we see that $\tilde{h}_\eps(t,x)=h_\eps(t,x)-\tfrac{\beta_\eps R(0)}{2\eps^2}t$, which implies 
 \[
 \E[h_\eps(t,x)]=\tfrac{\beta_\eps R(0)}{2\eps^2}t+O(|\log \eps|^{1/2})
 \]
 and matches \cite[Lemma 7.4]{chatterjee2018constructing}. On the other hand, it was shown in \cite{caravenna2015universality} that $\E[\log u_\eps]\to  -\frac12\log \frac{2\pi}{2\pi-\beta^2}$, thus a more precise expansion takes the form 
 \[
 \E[h_\eps(t,x)]=\tfrac{\beta_\eps R(0)}{2\eps^2}t+\beta_\eps^{-1}\E[ \log u_\eps]=\tfrac{\beta_\eps R(0)}{2\eps^2}t-\tfrac12\beta_\eps^{-1}\log \tfrac{2\pi}{2\pi-\beta^2}+o(|\log \eps|^{1/2}).
 \]

 \section{Variance convergence}\label{s.convar}

To simplify the notation, we define 
\[
M_{\eps,j}(t,x):=\exp\left(\beta_\eps\int_0^{t/\eps^2}V(s,\tfrac{x}{\eps}+B_s^j)ds-\tfrac{\beta_\eps^2R(0)t}{2\eps^2}\right),
\]
where $\{B^j\}_j$ are independent Brownian motions. For any set $I\subset \R_+, x\in\R^2$ and Brownian motions $B^i,B^j$, we define 
\begin{equation}\label{e.defcR}
\cR(I,x,B^i,B^j)=\int_IR(x+B^i_s-B^j_s)ds
\end{equation}
as the intersection time of $B^i,B^j$ during the interval $I$, and $x$ is the initial distance. For $I=[0,T]$, we write $\cR(T,x,B^i,B^j)=\cR([0,T],x,B^i,B^j)$. 

The following lemma will be used repeatedly and is taken from \cite[Lemma 3.1]{dunlap2018kpz}.

\begin{lemma}\label{l.holder}
For any $n\in\Z_+$ and $q>1$, there exists $\beta(n,q)>0$ such that if $\beta<\beta(n,q)$, then for any random variable $F(B^1,\ldots,B^n)\geq0$, $t>0$ and $\{x_j\in\R^2\}_{j=1,\ldots,n}$, we have
\begin{equation}\label{e.keyes}
\E\left[ \frac{\E_B[ \prod_{j=1}^n M_{\eps,j}(t,x_j) F(B^1,\ldots,B^n)]}{\prod_{j=1}^nZ_\eps(t,x_j)}\right] \les \E_B[ F(B^1,\ldots,B^n)^q]^{1/q}.
\end{equation}
\end{lemma}
\begin{proof}
By the
Cauchy-Schwarz inequality and Proposition~\ref{p.negative}, the square of the  l.h.s. of \eqref{e.keyes} is bounded by 
\[
\begin{aligned}
&\E\left[\left|\E_B[ \prod_{j=1}^n M_{\eps,j}(t,x_j) F(B^1,\ldots,B^n)]\right|^2\right]\\
&=\E_B\E\left[ \prod_{j=1}^{2n} M_{\eps,j}(t,x_j) F(B^1,\ldots,B^n) F(B^{n+1},\ldots,B^{2n})\right],
\end{aligned}
\]
where $x_{j+n}=x_j$ for $j=1,\ldots,n$. Evaluating the expectation with respect to $\dot{W}$, we obtain 
\[
\E\left[\prod_{j=1}^{2n} M_{\eps,j}(t,x_j)\right]= \exp\left(\tfrac{\beta_\eps^2}{2}\sum_{j,k=1}^{2n}\1_{j\neq k} \cR(\tfrac{t}{\eps^2},\tfrac{x_j-x_k}{\eps},B^j,B^k)\right).
\]
With
$p=\frac{q}{q-1}$, Lemma~\ref{l.exmm} shows that the r.h.s. of the above expression has an $L^p$ norm that is bounded uniformly in $\eps$ and $x_j$, provided that $\beta$ is chosen small. We apply H\"older
inequality to complete the proof.
\end{proof}

\subsection{The analysis of $I_{1,\eps}$}
Recall that $I_{1,\eps}$ is defined in \eqref{e.defI1}.
\begin{lemma}\label{l.i1}
For $K_\eps=\tfrac{1}{\eps^2|\log\eps|^\alpha}$ with $\alpha>1$, we have $\E[I_{1,\eps}^2]\to 0$ as $\eps\to0$.
\end{lemma}

\begin{proof}
Writing
$I_{1,\eps}
=\int_0^{K_\eps}\int_{\R^2} \E[ \Y_{s,y}|\F_s] dW(s,y)$ for the  appropriate
$\Y_{s,y}$, we have by
It\^o's isometry that
\[
\E[I_{1,\eps}^2]=\int_0^{K_\eps}\int_{\R^2} \E[ |\E[ \Y_{s,y}|\F_s] |^2]dyds \leq \int_0^{K_\eps}\int_{\R^2} \E[\Y_{s,y}^2] dyds,
\]
and
\[
\E[\Y_{s,y}^2] = \int_{\R^{4}} g(x_1)g(x_2) \E\left[ \frac{\E_B[\prod_{j=1}^2 M_{\eps,j}(t,x_j)\Phi^\eps_{t,x_j,B^j}(s,y)]}{Z_\eps(t,x_1)Z_\eps(t,x_2)} \right]dx_1dx_2.
\]
Using the fact that 
\[
\int_0^{K_\eps}\int_{\R^2} \prod_{j=1}^2\Phi^\eps_{t,x_j,B^j}(s,y)dyds=\int_0^{K_\eps}  R(\tfrac{x_1-x_2}{\eps}+B^1_s-B^2_s)ds=\cR(K_\eps,\tfrac{x_1-x_2}{\eps},B^1,B^2),
\]
where we recall that 
$R(x)=\int \varphi(x+y)\varphi(y)dy$, we have 
\[
\E[I_{1,\eps}^2] \les \int_{\R^{4}} |g(x_1)g(x_2)| \E\left[ \frac{\E_B[\prod_{j=1}^2 M_{\eps,j}(t,x_j)\cR(K_\eps,\frac{x_1-x_2}{\eps},B^1,B^2)]}{Z_\eps(t,x_1)Z_\eps(t,x_2)} \right]dx_1dx_2.
\]
By Lemma~\ref{l.holder}, we have
\begin{equation}\label{e.301}
\begin{aligned}
\E[I_{1,\eps}^2]\les& \int_{\R^{4}} |g(x_1)g(x_2)| \sqrt{\E_B[\cR^2(K_\eps,\tfrac{x_1-x_2}{\eps},B^1,B^2)]}dx_1dx_2\\
\les& \sqrt{\int_{\R^4}|g(x_1)g(x_2)| \E_B[\cR^2(K_\eps,\tfrac{x_1-x_2}{\eps},B^1,B^2)]dx_1dx_2}.
\end{aligned}
\end{equation}
We apply Lemma~\ref{l.291} to deduce 
\[
\E[I_{1,\eps}^2] \les \sqrt{(1+\log K_\eps)\eps^2 K_\eps} \les \tfrac{1}{|\log\eps|^{(\alpha-1)/2}}\to0.
\]
The proof is complete.
\end{proof}

\subsection{The analysis of $I_{2,\eps}$}

Recall that $I_{2,\eps}$ is defined in \eqref{e.defI2}. \begin{lemma}\label{l.i2}
For $K_\eps=\tfrac{1}{\eps^2|\log\eps|^\alpha}$ with $\alpha>0$, we have $\E[I_{2,\eps}^2]\to0$ as $\eps\to0$.
\end{lemma}

\begin{proof}
By the same calculation as in the proof of Lemma~\ref{l.i1}, we have
\[
\E[I_{2,\eps}^2] \les \int_{\R^{4}} g(x_1)g(x_2)A_\eps(x_1,x_2)dx_1dx_2,
\]
with 
\[
\begin{aligned}
&A_\eps(x_1,x_2)\\
&=\E\left[\frac{\E_B[\prod_{j=1}^2 M_{\eps,j}(t,x_j)\cR([K_\eps,t/\eps^2],\tfrac{x_1-x_2}{\eps},B^1,B^2)] }{Z(K_\eps,x_1/\eps)Z(K_\eps,x_2/\eps)} \prod_{j=1}^2\left(\frac{Z(K_\eps,x_j/\eps)}{Z(t/\eps^2,x_j/\eps)}-1\right)\right].
\end{aligned}
\]
Applying Proposition~\ref{p.negative}, H\"older inequality,
and the fact that $Z(t,x)$ is stationary in $x$, we have 
\[
|A_\eps(x_1,x_2)| \les \|a\|_8\|b\|_2\|b\|_4,
\]
where we simply denoted 
\[
a=\E_B[\prod_{j=1}^2 M_{\eps,j}(t,x_j)\cR([K_\eps,t/\eps^2],\tfrac{x_1-x_2}{\eps},B^1,B^2)], \quad b=Z(K_\eps,0)-Z(t/\eps^2,0).
\]
First, we write $b^4=b^{2+\delta}b^{2-\delta}$ and apply H\"older inequality to derive
\[
\E[b^4]=\E[b^{2+\delta}b^{2-\delta}]\leq \E[b^{(2+\delta)p}]^{1/p}\E[b^2]^{1/q}
\]
with $p^{-1}+q^{-1}=1$ and $q=\tfrac{2}{2-\delta}$. Since $|b|\leq Z(K_\eps,0)+Z(t/\eps^2,0)$, we apply Lemma~\ref{l.exmm} to bound $\E[b^{(2+\delta)p}]^{1/p}$ by some constant (for those $\beta\ll1$ depending on $\delta$). Thus we obtain $\|b\|_4\les \|b\|_2^{\frac{2-\delta}{4}}$ and further applying Lemma~\ref{l.292} below yields 
\[
\|b\|_2\|b\|_4 \les \|b\|_2^{\frac{6-\delta}{4}} \les \frac{1}{|\log \eps|^{\frac34-\delta'}} 
\]
for some $\delta'>0$ that is sufficiently close to zero. To analyze $a$, to simplify the notation, we write $a_1(V,B^1,B^2)=\prod_{j=1}^2 M_{\eps,j}(t,x_j)$ and $a_2(B^1,B^2)=\cR([K_\eps,\tfrac{t}{\eps^2}],\tfrac{x_1-x_2}{\eps},B^1,B^2)$ so $a=\E_B[a_1(V,B^1,B^2)a_2(B^1,B^2)]$. We first write
\[
\begin{aligned}
\|a\|_8^8&=\E[|\E_B[a_1(V,B^1,B^2)a_2(B^1,B^2)]|^8]\\
&=\E_B\E[\prod_{j=1}^8 a_1(V,B^{2j-1},B^{2j})a_2(B^{2j-1},B^{2j})],\\
\end{aligned}
\]
where $B^j$ are independent Brownian motions, then we average with respect to $V$ and follow the same proof of Lemma~\ref{l.holder} to derive
\[
\|a\|_8^8\les\E_B[ \prod_{j=1}^8 |a_2(B^{2j-1},B^{2j})|^2]^{1/2}=|\E_B[ |a_2(B^1,B^2)|^2]|^4=\|a_2\|_2^8
\]
This implies 
\[
\begin{aligned}
&\int_{\R^{4}} |g(x_1)g(x_2)| \times \|a\|_8\, dx_1dx_2\\
&\les \sqrt{\int_{\R^{4}} |g(x_1)g(x_2)|\times \|\cR([K_\eps,\tfrac{t}{\eps^2}],\tfrac{x_1-x_2}{\eps},B^1,B^2)\|_2^2 dx_1dx_2} \les\sqrt{|\log \eps|},
\end{aligned}
\]
where the second $\les$ comes from Lemma~\ref{l.291}. The proof is completed by choosing $\delta'<\tfrac14$.
\end{proof}

\begin{lemma}\label{l.292}
Recall $K_\eps=\tfrac{1}{\eps^2|\log\eps|^\alpha}$ with $\alpha>0$. For any $\delta>0$, there exists $\beta(\delta)>0$ such that if $\beta<\beta(\delta)$, we have 
\[
\E[|Z(\tfrac{t}{\eps^2},0)-Z(K_\eps,0)|^2]\les \frac{1}{|\log \eps|^{1-\delta}}.
\]
\end{lemma}

\begin{proof}
By the second moment calculation, we have 
\[
\begin{aligned}
    \E[|Z(\tfrac{t}{\eps^2},0)-Z(K_\eps,0)|^2]&=\E[Z(\tfrac{t}{\eps^2},0)^2]-\E[Z(K_\eps,0)^2]\\
    &=\E_B\left[ e^{\beta_\eps^2\int_0^{t/\eps^{2}}R(B_{2s})ds}-e^{\beta_\eps^2 \int_0^{K_\eps} R(B_{2s})ds}\right].
\end{aligned}
\]
Applying the simple inequality $|e^x-e^y|\leq (e^x+e^y)|x-y|$, H\"older inequality and Lemma~\ref{l.exmm}, we have 
\begin{equation}\label{e.11271}
\E_B\left[ e^{\beta_\eps^2\int_0^{t/\eps^{2}}R(B_{2s})ds}-e^{\beta_\eps^2 \int_0^{K_\eps} R(B_{2s})ds}\right]\les \frac{1}{|\log\eps|}\bigg\|\int_{K_\eps}^{t/\eps^2}R(B_{2s})ds\bigg\|_q
\end{equation}
for any $q>1$ (provided that $\beta<\beta(q)$). To estimate the above $L^q$ norm, we note that 
\[
\bigg\|\int_{K_\eps}^{t/\eps^2}R(B_{2s})ds\bigg\|_1  \les \log |\log \eps|, \quad\quad \bigg\|\int_{K_\eps}^{t/\eps^2}R(B_{2s})ds\bigg\|_2   \les \sqrt{|\log \eps|(\log |\log\eps|)},
\]
with both estimates coming from the same proof of Lemma~\ref{l.291}. More precisely, for the first estimate, we simply write the expectation in terms of the convolution of $R$ with the heat kernel and an integration of the heat kernel in time in $[K_\eps,t/\eps^2]$ leads to the bound of $\log \tfrac{t}{\eps^2K_\eps}$. The analysis of the second estimate is the same and gives an upper bound $\sqrt{\log \tfrac{t}{\eps^2K_\eps}\log \tfrac{t}{\eps^2}}$. For $\theta\in(0,1)$ and $q=\tfrac{2}{2-\theta}$, by the $L^p-$interpolation inequality we have 
\begin{equation}\label{e.302}
\bigg\|\int_{K_\eps}^{t/\eps^2}R(B_{2s})ds\bigg\|_q \les \bigg\|\int_{K_\eps}^{t/\eps^2}R(B_{2s})ds\bigg\|_1^{1-\theta}\bigg\|\int_{K_\eps}^{t/\eps^2}R(B_{2s})ds\bigg\|_2^{\theta}\les |\log \eps|^\delta,
\end{equation}
provided that $\theta$ is chosen sufficiently close to zero (i.e., $q$ is sufficiently close to $1$). The proof is complete.
\end{proof}

\subsection{The analysis of $I_{3,\eps}$}
Recall that $I_{3,\eps}$ is defined in \eqref{e.defI3}. Using the fact that $\E[M(t/\eps^2,x/\eps)|\F_s]=M(s,x/\eps)$, we have 
that
\[
I_{3,\eps}=\int_{K_\eps}^{t/\eps^2}\int_{\R^2} \left(\int_{\R^2} \frac{g(x)}{Z(K_\eps,x/\eps)}\E_B[ M(s,x/\eps)\Phi^\eps_{t,x,B}(s,y)] dx\right) dW(s,y).
\]

For any $T>0,x_1,x_2\in\R^2$ and a standard $2$-dimensional 
Brownian motion $\bar{B}$, we define the deterministic function
\[
\cH_\eps(T,x_1,x_2)=\E_{\bar{B}}\left[ e^{\beta^2_\eps \int_0^T R(x_1+\bar{B}_{2s})ds} \big|\bar{B}_{2T}=x_2\right].
\]
We introduce the following notation: for any $x,y\in\R^2$, the expectation $\hat{\E}_{x,y}$ is defined as 
\[
\hat{\E}_{x,y}[F]=\E\left[\frac{\E_B[ M_1(K_\eps,x)M_2(K_\eps,y) F]}{Z(K_\eps,x)Z(K_\eps,y)}\right]
\]
for any random variable $F$. In particular, we will consider functional of 
\[
\X_{K_\eps}=B^1_{K_\eps}-B^2_{K_\eps},
\]
so
\[
\hat{\E}_{x,y}[F(\X_{K_\eps})]=\E\left[\frac{\E_B[ M_1(K_\eps,x)M_2(K_\eps,y) F(B^1_{K_\eps}-B^2_{K_\eps})]}{Z(K_\eps,x)Z(K_\eps,y)}\right].
\]
Note that we have omitted the dependence of the expectation $\hat{\E}_{x,y}$ on $\eps$ to simplify the notation.
 
The following three lemmas combine to show the convergence of 
\begin{equation}\label{e.coni3}
\E[I_{3,\eps}^2]\to \sigma_t^2
\end{equation} with $\sigma_t^2$ given in \eqref{e.defsigma}.

\begin{lemma}\label{l.i3re}
$\E[I_{3,\eps}^2]= \int_{0}^{t-\eps^2K_\eps} \G_\eps(s) ds$,
with 
\begin{equation}
\begin{aligned}
\G_\eps(s)=\int_{\R^{6}}& g(x-w)g(x) R(y) \\
&\times \hat{\E}_{-w/\eps,0}\left[G_{2s}(w+\eps y-\eps \X_{K_\eps})\cH_\eps(\tfrac{s}{\eps^2},y,\X_{K_\eps}-\tfrac{w}{\eps}-y)\right]dxdydw.
\end{aligned}
\end{equation}
\end{lemma}

\begin{lemma}\label{l.i3bd}
There exists $\beta_0>0$ so that there exists $\gamma\in (0,1)$ such that,  for all  $\beta<\beta_0$,
$\G_\eps(s)\les s^{-\gamma}$ for $s\in(0,t)$.
\end{lemma}

\begin{lemma}\label{l.i3lim}
For any $s\in(0,t)$, $\G_\eps(s)\to \nu_{\mathrm{eff}}^2 \int_{\R^{4}} g(x-w)g(x)G_{2s}(w)dwdx$, as $\eps\to0$.
\end{lemma}

The proof of Lemmas~\ref{l.i3re} and \ref{l.i3bd} is the same as \cite[Lemma 3.5, 3.6]{dunlap2018kpz}.

\begin{proof}[Proof of Lemma~\ref{l.i3lim}]
Recall that 
\[
\begin{aligned}
\G_\eps(s)=\int_{\R^{6}} &g(x-w)g(x) R(y) \\
&\times   \hat{\E}_{-w/\eps,0}\left[G_{2s}(w+\eps y-\eps \X_{K_\eps})\cH_\eps(\tfrac{s}{\eps^2},y,\X_{K_\eps}-\tfrac{w}{\eps}-y)\right]dxdydw.
\end{aligned}
\]
Since $s>0$ is fixed and the term  $\cH_\eps$ is uniformly bounded by Lemma~\ref{l.exmm}, the expectation in the above expression is bounded uniformly in $x,y,w,\eps$, so we only need to pass to the limit of the expectation for fixed $x,y,w\in\R^2$ and $w\neq0$. The proof is divided into three steps.

(i) We show that $\hat{\E}_{-w/\eps,0}[|G_{2s}(w+\eps y-\eps \X_{K_\eps})-G_{2s}(w)|]\to0$ as $\eps\to0$. Using the fact that $G_{2s}(\cdot)$ has bounded derivatives (for fixed $s>0$) so $|G_{2s}(x)-G_{2s}(y)|\les |x-y|$:
\[
|G_{2s}(w+\eps y-\eps \X_{K_\eps})-G_{2s}(w)|\les \eps|y|+\eps |\X_{K_\eps}|,
\]
it suffices to show $\hat{\E}_{-w/\eps,0}[|\eps \X_{K_\eps}|]\to0$. We apply Lemma~\ref{l.holder} to derive
\begin{equation}\label{e.se202}
\hat{\E}_{-w/\eps,0}[|\eps \X_{K_\eps}|] \les \sqrt{ \E_B[|\eps \X_{K_\eps}|^2]}=\sqrt{2\eps^2K_\eps}\to0.
\end{equation}

(ii) Define $s_\eps=\tfrac{s}{\eps^2|\log\eps|}$ and 
\[
  \tilde{\cH}_\eps=\E_{\bar B}\left[ e^{\beta^2_\eps \int_0^{s_\eps} R(y+{\bar B}_{2r})dr}\bigg|\bar B_{2s/\eps^2}=\X_{K_\eps}-\tfrac{w}{\eps}-y\right],
\]
we show that 
\begin{equation}\label{e.se201}
  \hat{\E}_{-w/\eps,0}[\tilde{\cH}_\eps]\to \tfrac{2\pi}{2\pi -\beta^2}, \quad \mbox{ as } \eps\to0.
\end{equation}
 We first note that $\tilde{\cH}_\eps$ can be written more explicitly by conditioning on $\bar B_{2s_\eps}$:
\[
\begin{aligned}
  \tilde{\cH}_\eps=&\E_{\bar B}\left[ e^{\beta^2_\eps\int_0^{s_\eps}R(y+\bar B_{2r})dr} \times  \frac{G_{2s(1-|\log \eps|^{-1})}(\eps \X_{K_\eps}-w-\eps y-\eps \bar B_{2s_\eps})}{G_{2s}(\eps \X_{K_\eps}-w-\eps y)}\right]\\
  =&\frac{1}{(1-|\log \eps|^{-1})} \E_{\bar B}\left[ e^{\beta_\eps^2\int_0^{s_\eps}R(y+\bar B_{2r})dr} e^{-\frac{(\eps \X_{K_\eps}-w-\eps y-\eps \bar B_{2s_\eps})^2}{4s(1-|\log\eps|^{-1})}} e^{\frac{(\eps \X_{K_\eps}-w-\eps y)^2}{4s}}\right].
\end{aligned}
\]
There are three factors inside the above expectation. By an application of Lemma~\ref{l.holder} again and the fact that $\eps^2K_\eps\to0$ as $\eps\to0$, we have
\begin{equation}\label{e.5101}
\limsup_{\eps\to0}\hat{\E}_{-w/\eps,0}[e^{\lambda |\eps \X_{K_\eps}|^2}]\les1, 
\end{equation}
for any $\lambda>0$. We also have $\E_{\bar B}[ e^{\lambda\beta_\eps^2\int_0^{s_\eps}R(y+\bar B_{2r})dr} ] \les1$ for $\beta<\beta(\lambda)$. Thus, by the same proof of (i) and applying H\"older's inequality, we can replace the second factor by $e^{-w^2/4s}$ with a negligible error. For the third factor, we use the inequality $|e^x-e^y|\leq (e^x+e^y)|x-y|$ so
\[
\big|e^{\frac{(\eps \X_{K_\eps}-w-\eps y)^2}{4s}}-e^{\frac{w^2}{4s}}\big|\les\big(e^{\frac{(\eps \X_{K_\eps}-w-\eps y)^2}{4s}}+e^{\frac{w^2}{4s}}\big)|\eps \X_{K_\eps}-\eps y|(|\eps \X_{K_\eps}-w-\eps y|+|w|).
\]
By the exponential moment bounds given in \eqref{e.5101} we can show the r.h.s. of the above display is small hence replace the third factor by $e^{w^2/4s}$ with a negligible error. In the end, we apply \cite[Theorem 1]{KR}, 
\[
\beta_\eps^2 \int_0^{s_\eps} R(y+\bar{B}_{2r})dr = \frac{\beta^2\log 2s_\eps}{2|\log\eps|} \frac{1}{\log 2s_\eps}\int_0^{2s_\eps} R(y+\bar{B}_{r})dr\Rightarrow \lambda_\beta\mathrm{Exp}(1), \quad \lambda_\beta =\frac{\beta^2}{2\pi}.
\]
Note that $\int_0^{s_\eps} R(y+\bar{B}_{2r})dr$ measures the ``local time'' of the planar Brownian motion near the origin, and the $2\pi$ factor in $\lambda_\beta$ comes from the two dimensional heat kernel.
Lemma~\ref{l.exmm} ensures the uniform integrability, and we pass to the limit to obtain \eqref{e.se201}.

(iii) We show that 
\begin{equation}\label{e.se203}
\hat{\E}_{-w/\eps,0}[|\cH_\eps(\tfrac{s}{\eps^2},y,\X_{K_\eps}-\tfrac{w}{\eps}-y)-\tilde{\cH}_\eps|]\to0
\end{equation}
as $\eps\to0$. For the fixed $w\neq0$, define the event $A_w:=\{|\eps \X_{K_\eps}|>w/2\}$. First, we have
\[
\begin{aligned}
&|\cH(\tfrac{s}{\eps^2},y,\X_{K_\eps}-\tfrac{w}{\eps}-y)-\tilde{\cH}_\eps|\\
&=|\cH(\tfrac{s}{\eps^2},y,\X_{K_\eps}-\tfrac{w}{\eps}-y)-\tilde{\cH}_\eps|\1_{A_w}+|\cH(\tfrac{s}{\eps^2},y,\X_{K_\eps}-\tfrac{w}{\eps}-y)-\tilde{\cH}_\eps|\1_{A_w^c}\\
&\les \1_{A_w}+|\cH(\tfrac{s}{\eps^2},y,\X_{K_\eps}-\tfrac{w}{\eps}-y)-\tilde{\cH}_\eps|\1_{A_w^c},
\end{aligned}
\]
where in the last step we applied \eqref{eq:bridgeconcl} (note that $s<t$) to bound $\tilde{\cH}_\eps\leq \cH(\tfrac{s}{\eps^2},y,\X_{K_\eps}-\tfrac{w}{\eps}-y)\les 1$. Thus 
\[
\begin{aligned}
&\hat{\E}_{-w/\eps,0}[|\cH(\tfrac{s}{\eps^2},y,\X_{K_\eps}-\tfrac{w}{\eps}-y)-\tilde{\cH}_\eps|]\\
&\les \hat{\E}_{-w/\eps,0}[\1_{A_w}]+\hat{\E}_{-w/\eps,0}[|\cH(\tfrac{s}{\eps^2},y,\X_{K_\eps}-\tfrac{w}{\eps}-y)-\tilde{\cH}_\eps|\1_{A_w^c}].
\end{aligned}
\]
The first term on the r.h.s. goes to zero as $\eps\to0$ by \eqref{e.se202}. For the second term, we have
\[
\begin{aligned}
&\hat{\E}_{-w/\eps,0}[|\cH(\tfrac{s}{\eps^2},y,\X_{K_\eps}-\tfrac{w}{\eps}-y)-\tilde{\cH}_\eps|\1_{A_w^c}]\\
&\leq \hat{\E}_{-w/\eps,0} \left[ \E_B\bigg[e^{\beta_\eps^2\int_0^{s/\eps^2}R(y+\bar{B}_{2r})dr}\beta_\eps^2 \int_{s_\eps}^{s/\eps^2}R(y+\bar{B}_{2r})dr\bigg|\bar{B}_{2s/\eps^2}=\X_{K_\eps}-\tfrac{w}{\eps}-y\bigg]\1_{A_w^c}\right].
\end{aligned}
\]
The conditional expectation can be bounded using Lemma~\ref{l.exmm}: 
\begin{equation}\label{e.294}
\begin{aligned}
&\E_B\bigg[e^{\beta_\eps^2\int_0^{s/\eps^2}R(y+\bar{B}_{2r})dr}\beta_\eps^2 \int_{s_\eps}^{s/\eps^2}R(y+\bar{B}_{2r})dr\bigg|\bar{B}_{2s/\eps^2}=\X_{K_\eps}-\tfrac{w}{\eps}-y\bigg]\\
&\les \frac{1}{|\log \eps|}\sqrt{ \E_B\bigg[  |\int_{s_\eps}^{s/\eps^2}R(y+\bar{B}_{2r})dr|^2 \bigg|\bar{B}_{2s/\eps^2}=\X_{K_\eps}-\tfrac{w}{\eps}-y\bigg]}.
\end{aligned}
\end{equation}
In the event $A_w^c$, we have $|\eps \X_{K_\eps}|\leq w/2$, thus $c_1w\leq |\eps \X_{K_\eps}-w-\eps y|\leq c_2w$ for some $c_1,c_2>0$ (note that $y$ is fixed). Recall that $s_\eps=\tfrac{s}{\eps^2|\log\eps|}$, we apply Lemma~\ref{l.295} to derive 
\[
\E_B\bigg[  |\int_{s_\eps}^{s/\eps^2}R(y+\bar{B}_{2r})dr|^2 \bigg|\bar{B}_{2s/\eps^2}=\X_{K_\eps}-\tfrac{w}{\eps}-y\bigg]\les |\log \eps|(\log |\log \eps|),
\]
uniformly in $|\eps\X_{K_\eps}|\leq w/2$, so we pass to the limit in \eqref{e.294}, then obtain \eqref{e.se203}. 

To summarize, we have 
\[
\begin{aligned}
  \G_\eps(s)\to& \tfrac{2\pi}{2\pi-\beta^2 }\int_{\R^{3d}} g(x-w)g(x)R(y)G_{2s}(w) dxdydw\\
  &=\nu_{\mathrm{eff}}^2 \int_{\R^{4}} g(x-w)g(x)G_{2s}(w)dxdw,
\end{aligned}
\]
which completes the proof.
\end{proof}

\subsection{Proof of Proposition~\ref{p.convar}}

Recall that $X_\eps-\E[X_\eps]=\beta_\eps(I_{1,\eps}+I_{2,\eps}+I_{3,\eps})$. We combine Lemmas~\ref{l.i1}, \ref{l.i2} and \eqref{e.coni3} to derive 
\[
\beta_\eps^{-2}\Var[X_\eps]=\E[|I_{1,\eps}+I_{2,\eps}+I_{3,\eps}|^2]\to \sigma_t^2.
\]

\section{Gaussianity}\label{s.gauss}

Recall the goal is to show 
\[
\E[\|DX_\eps\|_H^4]^{1/4} \E[\|D^2X_\eps\|_{\mathrm{op}}^4]^{1/4}=o(|\log\eps|^{-1}),  \mbox{ as } \eps\to0,
\]
where $X_\eps=\int_{\R^2} \log Z_\eps(t,x) g(x)dx$ and $H=L^2(\R^{2+1})$. Since 
\[
DX_\eps= \int_{\R^2} \frac{DZ_\eps(t,x)}{Z_\eps(t,x)}g(x)dx,
\]
we have
\[
\begin{aligned}
D^2X_\eps&=D\int_{\R^2} \frac{DZ_\eps(t,x)}{Z_\eps(t,x)} g(x)dx\\
&=\int_{\R^2} \frac{ Z_\eps (t,x)D^2Z_\eps(t,x)-DZ_\eps(t,x)\otimes DZ_\eps(t,x)}{Z^2_\eps(t,x)}g(x)dx.
\end{aligned}
\]

Using the Feynman-Kac representation \eqref{e.fkre},
\[
D^2Z_\eps(t,x)=\beta^2_\eps \E_B[ M_\eps(t,x)\Phi_{t,x,B}^\eps\otimes \Phi_{t,x,B}^\eps],
\]
so 
\[
Z_\eps(t,x)D^2Z_\eps(t,x)=\beta^2_\eps \E_{B}\left[\prod_{j=1}^2M_{\eps,j}(t,x)\Phi_{t,x,B^2}^\eps\otimes \Phi_{t,x,B^2}^\eps\right],
\]
and
\[
DZ_\eps(t,x)\otimes D Z_\eps(t,x)=\beta^2_\eps \E_{B}\left[\prod_{j=1}^2M_{\eps,j}(t,x)\Phi_{t,x,B^1}^\eps\otimes \Phi_{t,x,B^2}^\eps\right].
\]
Thus we can write 
\[
\begin{aligned}
D^2X_\eps=\beta^2_\eps\int_{\R^2}  \frac{ \E_{B}[\prod_{j=1}^2M_{\eps,j}(t,x)(\Phi_{t,x,B^2}^\eps-\Phi_{t,x,B^1}^\eps)\otimes \Phi_{t,x,B^2}^\eps]}{Z_\eps^2(t,x)} g(x)dx=\cP_2-\cP_1,
\end{aligned}
\]
 where 
\[
H\otimes H\ni\cP_k=\beta^2_\eps\int_{\R^2}  \frac{ \E_{B}[\prod_{j=1}^2M_{\eps,j}(t,x)\Phi_{t,x,B^k}^\eps\otimes \Phi_{t,x,B^2}^\eps]}{Z^2_\eps(t,x)} g(x)dx.
\]
Thus,
\[
\|D^2X_\eps\|_{\mathrm{op}}^4\les \|\cP_1\|_{\mathrm{op}}^4+\|\cP_2\|_{\mathrm{op}}^4,
\]
and we only need to estimate $\E[\|\cP_k\|_{\mathrm{op}}^4]$, $k=1,2$.

\subsection{The first derivative}
\begin{lemma}\label{l.1stde}
For any $\delta>0$, there exists $\beta(\delta)>0$ such that if $\beta<\beta(\delta)$, 
\[
\E[\|DX_\eps\|_H^4]^{1/4} \les |\log \eps|^{-\frac12+\delta}.
\]
\end{lemma}

\begin{proof}
A direct calculation gives
\[
\begin{aligned}
\|DX_\eps\|_H^4=\beta^4_\eps\int_{\R^{8}} \prod_{j=1}^4 \frac{g(x_j)}{Z_\eps(t,x_j)}  \E_{B}\left[ \prod_{j=1}^4 M_{\eps,j}(t,x_j) \cR(\tfrac{t}{\eps^2},\tfrac{x_1-x_2}{\eps},B^1,B^2)\cR(\tfrac{t}{\eps^2},\tfrac{x_3-x_4}{\eps},B^3,B^4)\right]dx,
\end{aligned}
\]
with $\cR$ defined in \eqref{e.defcR}. Taking the expectation and applying Lemma~\ref{l.holder}, we have
\[
\begin{aligned}
\beta_\eps^{-4}\E[\|DX_\eps\|_H^4] \les& \int_{\R^{8}} \prod_{j=1}^4 |g(x_j)| \E_B\left[\cR^q(\tfrac{t}{\eps^2},\tfrac{x_1-x_2}{\eps},B^1,B^2)\cR^q(\tfrac{t}{\eps^2},\tfrac{x_3-x_4}{\eps},B^3,B^4)\right]^{1/q}dx\\
\les& \left(\int_{\R^4} |g(x_1)g(x_2)|\E_B[ \cR^q(\tfrac{t}{\eps^2},\tfrac{x_1-x_2}{\eps},B^1,B^2)]dx\right)^{2/q}.
\end{aligned}
\]
We can view the factor $|g(x_1)g(x_2)|$ as a weight (without loss of generality assume $\int |g|=1$), so the integral
\[
\int_{\R^4} |g(x_1)g(x_2)|\E_B[ \cR^q(\tfrac{t}{\eps^2},\tfrac{x_1-x_2}{\eps},B^1,B^2)]dx
\]
can be viewed as an expectation of $\cR^q(\tfrac{t}{\eps^2},\tfrac{x_1-x_2}{\eps},B^1,B^2)$ with $x_1,x_2$ independently sampled from the density $|g|$. Therefore, by the $L^p-$interpolation inequality and arguing similarly as \eqref{e.302}, we have 
\begin{equation}\label{e.304}
\begin{aligned}
\big(\int_{\R^4} |g(x_1)g(x_2)|\E_B[ &\cR^q(\tfrac{t}{\eps^2},\tfrac{x_1-x_2}{\eps},B^1,B^2)]dx\big)^{1/q}\\
\les& \big(\int_{\R^4} |g(x_1)g(x_2)|\E_B[ \cR(\tfrac{t}{\eps^2},\tfrac{x_1-x_2}{\eps},B^1,B^2)]dx\big)^{1-\theta}\\
&\times \big(\int_{\R^4} |g(x_1)g(x_2)|\E_B[ \cR^2(\tfrac{t}{\eps^2},\tfrac{x_1-x_2}{\eps},B^1,B^2)]dx\big)^{\theta/2}
\end{aligned}
\end{equation}
for $\theta=2-\tfrac{2}{q}$. Applying Lemma~\ref{l.291}, we know that the first factor on the r.h.s. is uniformly bounded and the second factor  is bounded by $|\log \eps|^{\theta/2}$. Thus, 
\[
\E[\|DX_\eps\|_H^4] \les \tfrac{|\log \eps|^{\theta}}{|\log\eps|^2}.
\]
By choosing $q$ sufficiently close to $1$, we can make $\theta$ arbitrarily small, which completes the proof.
\end{proof}

\subsection{The second derivative}

To estimate $\|\cP_k\|_{\mathrm{op}}$, we use the contraction inequality \cite[Proposition 4.1]{nourdin2009second}, which says that 
\[
\|\cP_k\|_{\mathrm{op}}^4\leq \|\cP_k\otimes_1\cP_k\|_{H\otimes H}^2.
\]
Here $\cP_k\otimes_1\cP_k$ is the random element of $H\otimes H$ obtained as the contraction of the symmetric
random tensor $\cP_k$. Recall that $H=L^2(\R^{2+1})$, we can write the r.h.s. of the above inequality as 
\[
\begin{aligned}
&\|\cP_k\otimes_1\cP_k\|_{H\otimes H}^2\\
&=\int_{\R^{2+1}}\int_{\R^{2+1}} \left(\int_{\R^{2+1}}\cP_k(s_1,y_1,s',y')\cP_k(s_2,y_2,s',y')ds'dy'\right)^2 ds_1dy_1ds_2dy_2.
\end{aligned}
\]

\subsubsection{The case $k=1$.} A direct calculation gives 
\[
\begin{aligned}
&\cP_1\otimes_1\cP_1\\
&=\beta^4_\eps\int_{\R^{4}} \frac{\E_B[ \prod_{j=1}^4 M_{\eps,j}(t,x_j) \cR(\frac{t}{\eps^2},\frac{x-y}{\eps},B^1,B^3)\Phi^\eps_{t,x,B^2}\otimes \Phi^\eps_{t,y,B^4}]}{Z_\eps^2(t,x)Z_\eps^2(t,y)} g(x)g(y)dxdy,
\end{aligned}
\]
where we write $x_1=x_2=x,x_3=x_4=y$ to simplify the notations. Thus, 
\[
\begin{aligned}
\|\cP_1\otimes_1\cP_1\|_{H\otimes H}^2=\beta^8_\eps\int_{\R^{8}}&g(x)g(y)g(z)g(w)\left(\prod_{j=1}^8 Z_\eps(t,x_j)\right)^{-1} \\
&\times \E_B\left[ \prod_{j=1}^8 M_{\eps,j}(t,x_j)\prod_{(i,k)\in \mathcal{O}} \cR(\tfrac{t}{\eps^2},\tfrac{x_i-x_k}{\eps},B^i,B^k) \right]dxdydzdw,
\end{aligned}
\]
where $x_5=x_6=z, x_7=x_8=w$, and the set $\mathcal{O}$ is $\mathcal{O}=\{(1,3),(5,7),(2,6),(4,8)\}$.
\begin{lemma}\label{l.p1}
For any $\delta>0$, there exists $\beta(\delta)$ such that if $\beta<\beta(\delta)$,  
\[
\E[\|\cP_1\otimes_1\cP_1\|_{H\otimes H}^2]\les |\log\eps|^{-4+\delta}.
\]
\end{lemma}

\begin{proof}
Applying Lemma~\ref{l.holder}, we have 
\[
\begin{aligned}
&\beta_\eps^{-8}\E[\|\cP_1\otimes_1\cP_1\|_{H\otimes H}^2] \\
&\les \int_{\R^{8}} |g(x)g(y)g(z)g(w)| \,\E_B\big[\prod_{(i,k)\in \mathcal{O}} \cR^q(\tfrac{t}{\eps^2},\tfrac{x_i-x_k}{\eps},B^i,B^k)\big]^{1/q} dxdydzdw\\
&\les  \left(\int_{\R^{8}} |g(x)g(y)g(z)g(w)|\, \E_B\big[\prod_{(i,k)\in \mathcal{O}} \cR^q(\tfrac{t}{\eps^2},\tfrac{x_i-x_k}{\eps},B^i,B^k)\big] dxdydzdw\right)^{1/q}=:a_q
\end{aligned}
\]
for some $q>1$ and we used the simplified notation $a_q$. Again we apply the $L^p-$interpolation inequality as in \eqref{e.304}, with $\theta=2-\tfrac{2}{q}$, we have
\[
a_q\leq a_1^{1-\theta}a_2^\theta.
\]
Note that the process $(B^i,B^k)$ are independent for different pairs of $(i,k)\in \mathcal{O}$. Applying Lemma~\ref{l.291} yields
\[
\begin{aligned}
&\E_B\big[\prod_{(i,k)\in \mathcal{O}} \cR(\tfrac{t}{\eps^2},\tfrac{x_i-x_k}{\eps},B^i,B^k)\big] \les \prod_{(i,k)\in\mathcal{O}}(1+|\log |x_i-x_k|| ),\\
&\E_B\big[\prod_{(i,k)\in \mathcal{O}} \cR^2(\tfrac{t}{\eps^2},\tfrac{x_i-x_k}{\eps},B^i,B^k)\big] \les |\log \eps|^4\prod_{(i,k)\in\mathcal{O}}(1+|\log |x_i-x_k||).
\end{aligned}
\]
Thus we have $a_1\les 1$ and $a_2\les |\log\eps|^2$, which implies $a_q\les |\log \eps|^{2\theta}$. By choosing $q$ sufficiently close to $1$ so that $\theta$ is sufficiently close to $0$, we complete the proof. 
\end{proof}

\subsubsection{The case $k=2$.} In this case, 
\[
\cP_2= \beta^2_\eps\int_{\R^2} \frac{\E_B[ M_\eps(t,x)\Phi^\eps_{t,x,B}\otimes \Phi^\eps_{t,x,B}]}{Z_\eps(t,x)}g(x)dx,
\]
so 
\[
\begin{aligned}
&\cP_2\otimes_1\cP_2\\
&= \beta^4_\eps\int_{\R^{4}} \frac{\E_B[ \prod_{j=1}^2 M_{\eps,j}(t,x_j) \cR(\frac{t}{\eps^2},\frac{x_1-x_2}{\eps},B^1,B^2)\Phi^\eps_{t,x_1,B^1}\otimes \Phi^\eps_{t,x_2,B^2}]}{Z_\eps(t,x_1)Z_\eps(t,x_2)} g(x_1)g(x_2)dx_1dx_2,
\end{aligned}
\]
and
\[
\|\cP_2\otimes_1 \cP_2\|_{H\otimes H}^2=\beta^8_\eps \int_{\R^{8}} \prod_{j=1}^4 \frac{g(x_j)}{Z_\eps(t,x_j)} \E_B\big[ \prod_{j=1}^4 M_{\eps,j}(t,x_j) \prod_{(i,k)\in \tilde{\mathcal{O}}} \cR(\tfrac{t}{\eps^2},\tfrac{x_i-x_k}{\eps},B^i,B^k)\big]dx,
\]
with the set $\tilde{\mathcal{O}}=\{(1,2),(3,4),(1,3),(2,4)\}$.

\begin{lemma}\label{l.p2}
For any $\delta>0$, there exists $\beta(\delta)$ such that if $\beta<\beta(\delta)$, 
\[
\E[\|\cP_2\otimes_1 \cP_2\|_{H\otimes H}^2]\les |\log\eps|^{-3+\delta}.
\]
\end{lemma}

\begin{proof}
By Lemma~\ref{l.holder} and the fact that $g$ is compactly supported, we have 
\[
\begin{aligned}
\beta_\eps^{-8}\E[\|\cP_2\otimes_1 \cP_2\|_{H\otimes H}^2] \les &\int_{\R^{8}} \prod_{j=1}^4 |g(x_j)| \,\E_B\big[ \prod_{(i,k)\in \tilde{\mathcal{O}}} \cR^q(\tfrac{t}{\eps^2},\tfrac{x_i-x_k}{\eps},B^i,B^k)\big]^{1/q}dx\\
\les &\left(\int_{\R^{8}} \prod_{j=1}^4|g(x_j)|\,\E_B\big[   \prod_{(i,k)\in \tilde{\mathcal{O}}} \cR^q(\tfrac{t}{\eps^2},\tfrac{x_i-x_k}{\eps},B^i,B^k)\big] dx\right)^{1/q}=:a_q.
\end{aligned}
\]
Arguing in the same way as in the proof of Lemma~\ref{l.p1}, we have $a_q\leq a_1^{1-\theta}a_2^\theta$ with $\theta=2-\tfrac{2}{q}$. By Lemma~\ref{l.bin}, we know that $a_1\leq |\log\eps|$. For $a_2$, to simplify the notation we write
\[\prod_{(i,k)\in \tilde{\mathcal{O}}} \cR^2(\tfrac{t}{\eps^2},\tfrac{x_i-x_k}{\eps},B^i,B^k)=\cR_1^2\cR_2^2\cR_3^2\cR_4^2,
\]
with $\cR_j$ denoting $\cR(\tfrac{t}{\eps^2},\tfrac{x_i-x_k}{\eps},B^i,B^k)$ for different $(i,k)\in \tilde{\mathcal{O}}$. Applying H\"older inequality and Lemma~\ref{l.291}, we derive 
\[
\begin{aligned}
\E_B\big[\prod_{(i,k)\in \tilde{\mathcal{O}}} \cR^2(\tfrac{t}{\eps^2},\tfrac{x_i-x_k}{\eps},B^i,B^k)\big] \leq &\|\cR_1\|_4^2\|\cR_2\|_8^2\|\cR_3\|_{16}^2\|\cR_4\|_{16}^2\\
\les &|\log\eps|^{2(\frac34+\frac78+\frac{15}{16}+\frac{15}{16})}\prod_{(i,k)\in\tilde{\mathcal{O}}}(1+|\log|x_i-x_k||)^{\alpha_{i,k}}
\end{aligned}
\]
for some $\alpha_{i,k}>0$. After integration in $x_j$, we have $a_2\les|\log\eps|^{\frac34+\frac78+\frac{15}{16}+\frac{15}{16}}$. Thus, by choosing $q$ sufficiently close to $1$, the proof is complete.
\end{proof}

\begin{lemma}\label{l.bin}
Assume $0\leq f,h\in\C_c^\infty(\R^2)$, then
\[
 \int_{\R^{8}}  \E_B\left[ \prod_{j=1}^4f(x_j) \prod_{(i,k)\in \tilde{\mathcal{O}}} \int_0^{t/\eps^2} h(\tfrac{x_i-x_k}{\eps}+B^i_s-B^k_s)ds\right]dx \les |\log \eps|.
 \]
\end{lemma}

\begin{proof}
Without loss of generality, assume $h$ is even. By symmetry, we assume in the proof that $\tilde{\mathcal{O}}=\{(1,2),(3,4),(2,3),(1,4)\}$. In this way we can write
\[
\begin{aligned}
\prod_{(i,k)\in \tilde{\mathcal{O}}} \int_0^{t/\eps^2} h(\tfrac{x_i-x_k}{\eps}+B^i_s-B^k_s)ds=\int_{[0,t/\eps^2]^4}\prod_{j=1}^4 h(\tfrac{x_j-x_{j-1}}{\eps}+B^j_{s_j}-B^{j-1}_{s_j})ds,
\end{aligned}
\]
with the convention $x_0=x_4,B^0=B^4$. Denoting $\hat{f}(\xi)=\int f(x)e^{-i\xi\cdot x}dx$ as the Fourier transform of $f$, we have 
\[
\begin{aligned}
&\int_{\R^{8}} \prod_{j=1}^4 f(x_j)h(\tfrac{x_j-x_{j-1}}{\eps}+B^j_{s_j}-B^{j-1}_{s_j}) dx\\
&=\frac{1}{(2\pi)^{8}}\int_{\R^{16}} \prod_{j=1}^4 f(x_j)\hat{h}(\eta_j) e^{i\eta_j\cdot (x_j-x_{j-1})/\eps}e^{i\eta_j \cdot(B^j_{s_j}-B^{j-1}_{s_j})} d\eta dx\\
&=\frac{1}{(2\pi)^{8}} \int_{\R^{8}} \prod_{j=1}^4 \hat{f}(\tfrac{\eta_j-\eta_{j-1}}{\eps}) \hat{h}(\eta_j)  e^{i(\eta_j\cdot B^j_{s_j}-\eta_{j+1}\cdot B^j_{s_{j+1}})}d\eta,
\end{aligned}
\]
with $\eta_0=\eta_4,\eta_5=\eta_1, s_5=s_1$. Thus, it suffices to estimate 
\[
\begin{aligned}
&\int_{[0,t/\eps^2]^4}\int_{\R^{8}}\prod_{j=1}^4 \hat{f}(\tfrac{\eta_j-\eta_{j-1}}{\eps}) \hat{h}(\eta_j)  \E_B[e^{i(\eta_j\cdot B_{s_j}-\eta_{j+1}\cdot B_{s_{j+1}})}]d\eta ds\\
&=\int_{[0,t]^4} \int_{\R^8}\prod_{j=1}^4 \hat{f}(\eta_j-\eta_{j-1}) \hat{h}(\eps\eta_j)  \E_B[e^{i(\eta_j\cdot B_{s_j}-\eta_{j+1}\cdot B_{s_{j+1}})}]d\eta ds,
\end{aligned}
\]
where we changed variables $s_j\mapsto s_j/\eps^2, \eta_j\mapsto \eps \eta_j$ and used the scaling property of the Brownian motion. Without loss of generality, consider the set $A_1=\{(s_1,\ldots,s_4)\in[0,t]^4: s_1\geq s_j,j\neq 1\}$, it is clear that in $A_1$ we have 
\[
\prod_{j=1}^4 \E_B[e^{i(\eta_j\cdot B_{s_j}-\eta_{j+1}\cdot B_{s_{j+1}})}]\leq e^{-\frac12|\eta_1|^2(s_1-s_{2})},
\]
which implies
\[
\begin{aligned}
&\int_{A_1} \int_{\R^8}\prod_{j=1}^4 |\hat{f}(\eta_j-\eta_{j-1}) \hat{h}(\eps\eta_j)|\,  \E_B[e^{i(\eta_j\cdot B_{s_j}-\eta_{j+1}\cdot B_{s_{j+1}})}]d\eta ds \\
&\leq \int_{A_1} \int_{\R^8}e^{-\frac12|\eta_1|^2(s_1-s_{2})}\prod_{j=1}^4 |\hat{f}(\eta_j-\eta_{j-1}) \hat{h}(\eps\eta_j)|d\eta ds\\
&\les \int_{A_1}\int_{\R^8}e^{-\frac12|\eta_1|^2(s_1-s_{2})}|\hat{h}(\eps \eta_1)\hat{f}(\tilde{\eta}_2)\hat{f}(\tilde{\eta}_3)\hat{f}(\tilde{\eta}_4)|d\eta_1d\tilde{\eta} ds.
\end{aligned}
\]
In the last ``$\les$'' we bounded $|\hat{f}(\eta_1-\eta_4)|\les1$ and changed variables $\eta_j-\eta_{j-1}\mapsto \tilde{\eta}_j,j=2,3,4$. The last integral can be computed explicitly, and we use the fact that 
\[
\int_0^t \int_{\R^2}e^{-\frac12|\eta_1|^2s} |\hat{h}(\eps \eta_1)| d\eta_1 ds=\int_0^{t/\eps^2} \int_{\R^2} e^{-\frac12|\eta_1|^2 s} |\hat{h}(\eta_1)| d\eta_1 ds\les |\log\eps|
\]
to complete the proof.
\end{proof}

\subsection{Proof of Proposition~\ref{p.gauss}}
Recall that $Y_\eps=\tfrac{X_\eps-\E[X_\eps]}{\sqrt{\Var[X_\eps]}}$.
Since 
\[
\begin{aligned}
d_{\mathrm{TV}}(Y_\eps,\zeta) \les \E[\|DY_\eps\|_H^4]^{1/4} \E[\|D^2Y_\eps\|_{\mathrm{op}}^4]^{1/4}=\tfrac{1}{\Var[X_\eps]}\E[\|DX_\eps\|_H^4]^{1/4} \E[\|D^2X_\eps\|_{\mathrm{op}}^4]^{1/4},
\end{aligned}
\]
using the fact that $\Var[X_\eps]\sim |\log \eps|^{-1}$ and  applying Lemmas~\ref{l.1stde}, \ref{l.p1} and \ref{l.p2}, we have 
\[
d_{\mathrm{TV}}(Y_\eps,\zeta) \les |\log\eps|\times |\log \eps|^{-\frac12+\delta}\times \left(|\log \eps|^{-4+\delta} +|\log \eps|^{-3+\delta}\right)^{\frac14}.
\]
By choosing $\delta$ small, the r.h.s. goes to zero as $\eps\to0$.

\appendix

\section{Auxiliary lemmas}

Recall that $\beta_\eps=\beta|\log\eps|^{-1/2}$ and $G_t(x)$ is the standard heat kernel. We present some rather standard estimates on the integral of Brownian functionals for the convenience of readers. Similar estimates can be found in \cite{CRS18,chatterjee2018constructing}.
\begin{lemma}\label{l.exmm}
Fix $t>0$, there exists $\beta_0>0$ such that if $\beta<\beta_0$, we have in $d=2$ that  
\begin{equation}
\label{eq:BMconcl}
\sup_{x\in\R^2,\eps\in(0,1)} \E_B\left[ e^{\beta_\eps^2\int_0^{t/\eps^2} R(x+B_s)ds}\right]<\infty,
\end{equation}
and 
\begin{equation}
\label{eq:bridgeconcl}
\sup_{x,y\in\R^2, \eps\in(0,1)} \E_B\left[ e^{\beta_\eps^2\int_0^{t/\eps^2} R(x+B_s)ds} \bigg| B_{t/\eps^2}=y\right]<\infty.
\end{equation}
\end{lemma}
\begin{proof}
It suffices to prove \eqref{eq:bridgeconcl} since \eqref{eq:BMconcl} follows from an integration in $y$. We claim \eqref{eq:bridgeconcl} is implied by
\begin{equation}\label{e.291}
\sup_{x,y\in\R^2,\eps\in(0,1)} |\log\eps|^{-1}\int_0^{t/\eps^2}\E_B[R(x+B_s)|B_{t/\eps^2}=y]ds \leq C.
\end{equation}
The proof essentially follows Portenko's lemma and we sketch it here for the convenience of readers. For any $n\geq 1$, we can write 
\[
\E_B\left[\left(\int_0^{t/\eps^2}R(x+B_s)ds\right)^n\bigg| B_{t/\eps^2}=y\right]=n!\int_{[0,t/\eps^2]_<^n} \E_B\left[ \prod_{j=1}^nR(x+B_{s_j}) \bigg|B_{t/\eps^2}=y\right]ds.
\]
By conditioning on $B_{s_{n-1}}$ and applying \eqref{e.291} to the integral in $s_n$, we have
\[
\int_{[0,t/\eps^2]_<^n} \E_B\left[ \prod_{j=1}^nR(x+B_{s_j}) \bigg|B_{t/\eps^2}=y\right]ds \leq C|\log \eps|\int_{[0,t/\eps^2]_<^{n-1}} \E_B\left[ \prod_{j=1}^{n-1}R(x+B_{s_j})\right]ds.
\]
Iterating this procedure yields
\[
\E_B\left[\left(\int_0^{t/\eps^2}R(x+B_s)ds\right)^n\bigg| B_{t/\eps^2}=y\right] \leq n! (C|\log\eps|)^n,
\]
which completes the proof of \eqref{eq:bridgeconcl}.

It remains to show \eqref{e.291}. For the Brownian bridge, we only need to show 
\[
\int_0^{t/2\eps^2} \E_B[R(x+B_s)|B_{t/\eps^2}=y]ds \leq C|\log\eps|
\]
for some constant $C$ independent of $x,y,\eps$. Note that $B_s$ has the Gaussian distribution with mean $\tfrac{ys}{t/\eps^2}$ and variance $\tfrac{s(t/\eps^2-s)}{t/\eps^2}$, so its density is bounded from above by $\frac{1}{s}$ when $s\leq \tfrac{t}{2\eps^2}$. Thus, we have 
\[
 \int_0^{t/2\eps^2} \E_B[R(x+B_s)|B_{t/\eps^2}=y]ds\les 1+\int_1^{t/2\eps^2}\int_{\R^2} R(x+w)\tfrac{1}{s} dwds \les 1+\log\tfrac{t}{2\eps^2}.
\]
The proof is complete.
\end{proof}

\begin{lemma}\label{l.291}
For any $0\leq g\in \C_c(\R^2)$ and $t>1$, we have
\begin{equation}\label{e.basicinequality}
\begin{aligned}
&\int_{\R^4} g(x_1)g(x_2)\E_B\left[\big|\int_0^t R(\tfrac{x_1-x_2}{\eps}+B_s)ds\big|^2\right]dx_1dx_2\leq C(1+\log t)\eps^2 t,\\
&\int_{\R^4} g(x_1)g(x_2)\E_B\left[\int_0^t R(\tfrac{x_1-x_2}{\eps}+B_s)ds\right]dx_1dx_2\leq C\eps^2 t,
\end{aligned}
\end{equation}
with some constant $C$ independent of $t,\eps$. For any $t>0,n\in\mathbb{Z}_+$, we also have 
\begin{equation}\label{e.pointwiseinequality}
\begin{aligned}
\E_B\left[\big|\int_0^{t/\eps^{2}} R(\tfrac{x}{\eps}+B_s)ds\big|^n\right]\leq C(n,t)|\log \eps|^{n-1}(1+|\log |x||)
\end{aligned}
\end{equation}
with some constant $C(n,t)$ independent of $x,\eps$.
\end{lemma}
\begin{proof}
To prove \eqref{e.basicinequality}, we write the expectation explicitly:
\begin{equation}\label{e.293}
\begin{aligned}
\E_B\left[\big|\int_0^t R(\tfrac{x}{\eps}+B_s)ds\big|^2\right]=&2\int_0^t ds\int_0^s du \, \E_B[ R(\tfrac{x}{\eps}+B_s)R(\tfrac{x}{\eps}+B_u)]\\
=&2\int_0^t ds\int_0^s du \int_{\R^4} R(y)R(z)G_u(z-\tfrac{x}{\eps})G_{s-u}(y-z)dzdy.
\end{aligned}
\end{equation}
Integrating in $s$ and $y$ yields
\begin{equation}\label{e.510}
\int_u^t \int_{\R^2} R(y)G_{s-u}(y-z)dyds\les 1+\int_{u+1}^t \int_{\R^2} R(y)\tfrac{1}{s-u}dyds \les 1+\log t,
\end{equation}
which implies
\begin{equation}\label{e.basic1}
\E_B\left[\big|\int_0^t R(\tfrac{x}{\eps}+B_s)ds\big|^2\right] \les (1+\log t) \int_0^t  \int_{\R^2} R(z) G_u(z-\tfrac{x}{\eps})dzdu.
\end{equation}
Since the integral $\int_0^t\int_{\R^2}R(z)G_u(z-\tfrac{x}{\eps})dzdu=\E_B[\int_0^t R(\tfrac{x}{\eps}+B_s)ds]$, to prove \eqref{e.basicinequality}, we only need to note that 
\[
\begin{aligned}
&\int_{\R^4}g(x_1)g(x_2)\int_0^t du \int_{\R^2} R(z) G_u(z-\tfrac{x_1-x_2}{\eps})dz \\
&= \tfrac{\eps^2}{(2\pi)^2}\int_0^t \int_{\R^2} |\hat{g}(\xi)|^2 \hat{R}(\eps \xi) e^{-\frac12|\eps\xi|^2 u} d\xi du \les \eps^2 t,
\end{aligned}
\]
where we used the fact that $\sup_{\xi\in\R^d}\hat{R}(\xi) \leq \int R=1$ in the last step.

To prove \eqref{e.pointwiseinequality}, by the same argument above, we have 
\[
\E_B\left[\big|\int_0^{t/\eps^{2}} R(\tfrac{x}{\eps}+B_s)ds\big|^n\right]\les |\log \eps|^{n-1}\int_0^{t/\eps^2}  \int_{\R^2} R(z) G_u(z-\tfrac{x}{\eps})dzdu.
\]
We estimate the integral in $u$ by 
\[
\int_0^t G_u(z-\tfrac{x}{\eps}) du\les \int_0^t u^{-1}e^{- \frac{|z-x/\eps|^2}{2u}}du \les \int_{\frac{|z-x/\eps|^2}{2t}}^\infty \lambda^{-1}e^{-\lambda}d\lambda \les 1+|\log|\eps z-x||+|\log \eps^2 t|.
\]
For the integral in $z$, recall that $R(z)=0$ for $|z|\geq 1$, we have 
\[
\int_{\R^2}R(z)|\log|\eps z-x|| dz\les \int_{|z|\leq 1}|\log|\eps z-x||dz \les 1+ |\log |x||+\1_{|x|\leq 3\eps}|\log\eps|\les 1+|\log |x||.
\]
The proof is complete.
\end{proof}

\begin{lemma}\label{l.295}
Fix $t>0$ and a compact set $K\subset \R^2$ with $0\notin K$, we have 
\[
\sup_{w\in K}\,\E\left[ \big(\int_{\tfrac{t}{\eps^2|\log\eps|}}^{\tfrac{t}{\eps^2}}R(B_s)ds\big)^2\bigg|B_{t/\eps^2}=\frac{w}{\eps}\right]\les |\log\eps|(\log|\log \eps|).
\]
\end{lemma}

\begin{proof}
To simplify the notation, denote $t_1=\tfrac{t}{\eps^2|\log\eps|}$ and $t_2=\tfrac{t}{\eps^2}$, and write 
\[
\big(\int_{t_1}^{t_2} R(B_s)ds\big)^2=2\int_{[t_1,t_2]_<^2}R(B_s)R(B_u)\1_{s>u}duds.
\]
Now we compute the conditional expectation
\[
\begin{aligned}
\E[R(B_s)R(B_u)|B_{t_2}=w/\eps]=&\int_{\R^{4}} R(x)R(y)\frac{G_{u}(y)G_{s-u}(x-y)G_{t_2-s}(\tfrac{w}{\eps}-x)}{G_{t_2}(\tfrac{w}{\eps})} dxdy\\
=&\int_{\R^{4}} R(x)R(y) \frac{G_{u}(y)G_{s-u}(x-y)G_{t-\eps^2 s}(w-\eps x)}{G_{t}(w)} dxdy.
\end{aligned}
\]
Since $0\notin K$, $t>0$ is fixed, and $\mathrm{supp}(R)\subset \{x:|x|\leq 1\}$, we have for $\eps\ll1$ that
\[
\sup_{w\in K, \,|x|\leq 1}\frac{G_{t-\eps^2s}(w-\eps x)}{G_t(w)}\les 1,
\]
uniformly in $s\leq t/\eps^2$, so the conditional expectation is bounded by 
\[
\E[R(B_s)R(B_u)|B_{t_2}=w/\eps]\les \int_{\R^{4}} R(x)R(y)G_u(y)G_{s-u}(x-y)dxdy,
\]
which implies
\[
\begin{aligned}
&\E\left[ \big(\int_{\tfrac{t}{\eps^2|\log\eps|}}^{\tfrac{t}{\eps^2}}R(B_s)ds\big)^2\bigg|B_{t/\eps^2}=w/\eps\right]\\
&\les \int_{[t_1,t_2]_<^2}\int_{\R^{4}} R(x)R(y)G_u(y)G_{s-u}(x-y)dxdyduds.
\end{aligned}
\]
By \eqref{e.510}, the above integral bounded by 
\[
\log (t_2-t_1)\log\tfrac{t_2}{t_1} \les |\log \eps|(\log|\log\eps|),
\]
which completes the proof.
\end{proof}

\section{Negative moments of $Z_\eps(t,x)$}\label{s.nemm}
The goal is to show there exists $\beta_0>0$ such that if $\beta<\beta_0$ and $n\in\Z_+$, we have
\begin{equation}\label{e.negmm}
\sup_{t\in[0,T]}\sup_{\eps\in(0,1)} \E[ Z_\eps(t,x)^{-n}] \leq C_{\beta,n,T}
\end{equation}
for some constant $C_{\beta,n,T}>0$. The result is essentially implied by \cite[Theorem 4.6]{hu2018asymptotics}, and we only present the details here for the convenience of the readers. Since $Z_\eps(t,x)$ has the same distribution as $u(\tfrac{t}{\eps^2},\tfrac{x}{\eps})$ and is stationary in the $x-$variable, it suffices to estimate the small ball probability $\Pb[u(\tfrac{t}{\eps^2},x)\leq r]$ for $r\ll1$. From now on, we will fix $\eps>0$ and derive an estimate that is uniform in $\eps>0$ and $t\in[0,T]$. We fix $t>0,x\in\R^2$.

We first define an approximation of the spacetime white noise
\[
\dot{W}_\delta(t,x)=e^{-\delta(t^2+|x|^2)}\int_{\R^{3}} \phi_\delta(t-s,x-y)dW(s,y),
\]
where $\phi_\delta(t,x)=\tfrac{1}{\delta^4}\phi(\tfrac{t}{\delta^2},\tfrac{x}{\delta})$ with $\phi\in\C_c^\infty(\R^{3})$ such that $\phi$ is even and $\int \phi=1$. Thus, we have almost surely that $\dot{W}_\delta\in L^2(\R^{3})\cap \C^\infty(\R^{3})$. Define 
\[
V_\delta(t,x)=\int_{\R^2} \varphi(x-y)\dot{W}_\delta(t,y)dy, \    \   \ccR_\delta(t,s,x,y)=\E[V_\delta(t,x)V_\delta(s,y)],
\]
 and $\cU_{\eps,\delta}(t,x)=\E_B\left[ e^{\V_{\eps,\delta}(t,B) }\right]$,
with 
\[
\V_{\eps,\delta}(t,B)=\beta_\eps\int_0^{t/\eps^2} V_\delta(\tfrac{t}{\eps^2}-s,x+B_s)ds-\tfrac12\beta^2_\eps\Q_\delta(\tfrac{t}{\eps^2},x,x,B,B),
\]
where
\[
\Q_\delta(t,x,y,B^1,B^2)=\int_{[0,t]^2}  \ccR_\delta(t-s,t-\ell,x+B_s^1,y+B_{\ell}^2)dsd\ell.
\]
By \cite[Proposition 4.2]{hu2018asymptotics}, for each fixed $\eps>0$, $\cU_{\eps,\delta}(t,x)\to u(\tfrac{t}{\eps^2},x)$ in probability as $\delta\to0$, so we only need to estimate $\Pb[\cU_{\eps,\delta}(t,x)\leq r]$ for $r\ll1$, uniformly in $\eps,\delta>0$ and $t\in[0,T]$.

With any given $\dot{W}_\delta$, define the expectation
\[
\E^{\dot{W}_\delta}_B[ F(B^1,B^2)]=\frac{\E_B[ F(B^1,B^2)e^{\V_{\eps,\delta}(t,B^1)+\V_{\eps,\delta}(t,B^2)}]}{\E_B[ e^{\V_{\eps,\delta}(t,B^1)+\V_{\eps,\delta}(t,B^2)}]}.
\]
To emphasize the dependence of $\cU_{\eps,\delta}$ on $\dot{W}_\delta$, we write $\cU_{\eps,\delta}(t,x)=\cU_\eps(t,x,\dot{W}_\delta)$. For any $\lambda>0$, define the set 
\[
A_\lambda(t,x)=\left\{\dot{W}_\delta: \cU_\eps(t,x,\dot{W}_\delta)>\tfrac12, \quad \beta_\eps^2\int_0^{t/\eps^2}\E_B^{\dot{W}_\delta}[ R(B^1_s-B^2_s)]ds\leq \lambda\right\}.
\]

\begin{lemma}
For any $\dot{\W}_\delta\in A_\lambda(t,x)$, we have 
\[
\cU_\eps(t,x,\dot{W}_\delta)\geq \tfrac12e^{-\sqrt{\lambda} \|\dot{W}_\delta-\dot{\W}_\delta\|_{L^2(\R^3)}}.
\]
\end{lemma}

\begin{proof}
We write 
\[
\begin{aligned}
\cU_\eps(t,x,\dot{W}_\delta)=\E_B[ e^{\V_{\eps,\delta}(t,B)}]=&\E_B[e^{\cV_{\eps,\delta}(t,B)}] \frac{ \E_B[e^{\V_{\eps,\delta}(t,B)-\cV_{\eps,\delta}(t,B)} e^{\cV_{\eps,\delta}(t,B)}]}{\E_B[ e^{\cV_{\eps,\delta}(t,B)}]}\\
=&\cU_\eps(t,x,\dot{\W}_\delta) \E_B^{\dot{\W}_\delta}[ e^{\V_{\eps,\delta}(t,B)-\cV_{\eps,\delta}(t,B)}],
\end{aligned}
\]
where $\cV_{\eps,\delta}(t,B)$ is obtained by replacing $\dot{W}_\delta$ by $\dot{\W}_\delta$ in the expression of $\V_{\eps,\delta}(t,B)$. By the fact that $\dot{\W}_\delta\in A_\lambda$ and Jensen's inequality, we have 
\[
\cU_\eps(t,x,\dot{W}_\delta)\geq\tfrac12  \exp(\E_B^{\dot{\W}_\delta}[ \V_{\eps,\delta}(t,B)-\cV_{\eps,\delta}(t,B)]).
\]
It remains to show  that
\begin{equation}\label{e.se251}
|\E_B^{\dot{\W}_\delta}[ \V_{\eps,\delta}(t,B)-\cV_{\eps,\delta}(t,B)]| \leq \sqrt{\lambda}\|W_\eps-\tilde{W}_\eps\|_{L^2(\R^3)}.
\end{equation}
We write 
\[
\begin{aligned}
\V_{\eps,\delta}(t,B)-\cV_{\eps,\delta}(t,B)
=\beta_\eps \int_0^{t/\eps^2}\int_{\R^2} \varphi(x+B_s-y)[\dot{W}_\delta(\tfrac{t}{\eps^2}-s,y)-\dot{\W}_\delta(\tfrac{t}{\eps^2}-s,y)]dyds,
\end{aligned}
\]
and apply Cauchy-Schwarz to derive 
\[
\begin{aligned}
&|\E_B^{\dot{\W}_\delta}[ \V_{\eps,\delta}(t,B)-\cV_{\eps,\delta}(t,B)]|\\
&\leq \|\dot{W}_\delta-\dot{\W}_\delta\|_{L^2(\R^3)}\sqrt{\beta_\eps^2\int_0^{t/\eps^2} \int_{\R^2} |\E_B^{\dot{\W}_\delta}[\varphi(x+B_s-y)] |^2 dyds} \\
&= \|\dot{W}_\delta-\dot{\W}_\delta\|_{L^2(\R^3)}\sqrt{\beta_\eps^2\int_0^{t/\eps^2} \E_B^{\dot{\W}_\delta}[R(B^1_s-B^2_s)]ds}\leq \sqrt{\lambda} \|\dot{W}_\delta-\dot{\W}_\delta\|_{L^2(\R^3)},
\end{aligned}
\]
which completes the proof.
\end{proof}

\begin{lemma}\label{l.lowerbd}
There exists constants $\lambda,c>0$ independent of $\eps,\delta>0$ and $t\in[0,T]$ such that $\Pb[A_\lambda(t,x)]\geq c$.
\end{lemma}

\begin{proof}
We have 
\[
\Pb[A_\lambda(t,x)] \geq \Pb[\cU_\eps(t,x,\dot{W}_\delta)>\tfrac{1}{2}]-\Pb[ B_\lambda(t,x)],
\]
with 
\[
B_\lambda(t,x)=\left\{\dot{W}_\delta:  \cU_\eps(t,x,\dot{W}_\delta)>\tfrac12,  \beta_\eps^2\int_0^{t/\eps^2}\E_B^{\dot{W}_\delta}[ R(B^1_s-B^2_s)]ds> \lambda\right\}.
\]
Using the fact that $\E[\cU_\eps(t,x,\dot{W}_\delta)]=1$ and the Paley-Zygmund's inequality, we have 
\[
\Pb[\cU_\eps(t,x,\dot{W}_\delta)>\tfrac{1}{2}]\geq \frac{1}{4\E[\cU_\eps(t,x,\dot{W}_\delta)^2]}=\frac{1}{4\E_B[ e^{\beta_\eps^2 \Q_\delta(t/\eps^2,x,x,B^1,B^2)}]}.
\]
For $B_\lambda(t,x)$, we have 
\[
\begin{aligned}
\Pb[B_\lambda(t,x)]\leq &\Pb\left[ \beta_\eps^2\int_0^{t/\eps^2} \E_B[ R(B_s^1-B_s^2) e^{\V_{\eps,\delta}(t,B^1)+\V_{\eps,\delta}(t,B^2)}] ds >\tfrac{\lambda}{4}\right] \\
\leq &\tfrac{4}{\lambda}\E_B\left[ e^{\beta_\eps^2 \Q_\delta(t/\eps^2,x,x,B^1,B^2)}\beta_\eps^2\int_0^{t/\eps^2} R(B^1_s-B^2_s)ds\right]\\
\leq &\tfrac{4C}{\lambda}\E_B\left[ e^{2\beta^2_\eps \Q_\delta(t/\eps^2,x,x,B^1,B^2)}\right]^{1/2}
\end{aligned}
\]
for some constant $C>0$, where the last ``$\leq$'' comes from an application of Cauchy-Schwarz inequality and Lemma~\ref{l.exmm}. By Lemma~\ref{l.intersection} and choosing $\lambda$ large, there exists some constants $c,\lambda>0$ independent of $\eps,\delta>0$ such that $\Pb[A_\lambda(t,x)] \geq c$, which completes the proof.
\end{proof}

\begin{lemma}\label{l.intersection}
There exists $\beta_0>0$ such that if $\beta<\beta_0$, we have 
\[
1\leq\sup_{t\in[0,T]}\,\sup_{\eps,\delta\in(0,1)}\E_B\left[ e^{\beta_\eps^2 \Q_\delta(t/\eps^2,x,x,B^1,B^2)}\right]\leq C_{\beta,T}.
\]
\end{lemma}

\begin{proof}
Recall that $\Q_\delta(t,x,x,B^1,B^2)=\int_{[0,t]^2}  \ccR_\delta(t-s,t-\ell,x+B_s^1,x+B_{\ell}^2)dsd\ell$. We write $\ccR_\delta$ explicitly:
\[
\begin{aligned}
\ccR_\delta(t_1,t_2,x_1,x_2)=&\int_{\R^{4}} \varphi(x_1-y_1)\varphi(x_2-y_2)\E[\dot{W}_\delta(t_1,y_1)\dot{W}_\delta(t_2,y_2)]dy_1dy_2\\
\leq &\int_{\R^{4}} \varphi(x_1-y_1)\varphi(x_2-y_2) \phi_\delta\star\phi_\delta(t_1-t_2,y_1-y_2) dy_1dy_2,
\end{aligned}
\]
with ``$\star$'' denoting the convolution. By the fact that $\varphi,\phi$ have compact supports, it is clear that 
\[
\ccR_\delta(t_1,t_2,x_1,x_2) \les \delta^{-2}  \1_{|x_1-x_2|\leq c,|t_1-t_2|\leq c\delta^2}
\]
for some $c>0$. Thus, we have 
\[
\begin{aligned}
\Q_\delta(t/\eps^2,x,x,B^1,B^2)\les &\int_{[0,t/\eps^2]^2} \delta^{-2}\1_{|s-\ell|\leq c\delta^2} \1_{|B^1_s-B^2_\ell|\leq c} dsd\ell\\
&\les  \int_0^{c} \left(\int_0^{t/\eps^2} \1_{|B^1_{\ell+\delta^2 s}-B^2_\ell|\leq c}d\ell\right)ds.
\end{aligned}
\]
By Jensen's inequality, we have 
\[
\begin{aligned}
\E_B[e^{\beta_\eps^2 \Q_\delta(t/\eps^2,x,x,B^1,B^2)}] \leq &\E_B\left[ \exp\bigg(c'\int_0^c \big(\beta_\eps^2\int_0^{t/\eps^2} \1_{|B^1_{\ell+\delta^2 s}-B^2_\ell|\leq c}d\ell\big)ds\bigg)\right]\\
\leq &\tfrac{1}{c}\int_0^c \E_B[\exp\big(cc'\beta_\eps^2\int_0^{t/\eps^2} \1_{|B^1_{\ell+\delta^2 s}-B^2_\ell|\leq c}d\ell\big)] ds,
\end{aligned}
\]
for some $c,c'>0$. Clearly we have 
\[
\begin{aligned}
&\sup_{s\in[0,c]}\E_B[\exp\big(cc'\beta_\eps^2\int_0^{t/\eps^2} \1_{|B^1_{\ell+\delta^2 s}-B^2_\ell|\leq c}d\ell\big)]\\
& \leq \sup_{x\in\R^2} \E_B[ \exp\big(cc'\beta_\eps^2\int_0^{t/\eps^2} \1_{|x+B^1_{\ell}-B^2_\ell|\leq c}d\ell\big)]\les 1
\end{aligned}
\]
for small $\beta$, where the last ``$\les$'' comes from Lemma~\ref{l.exmm}. The proof is complete.
\end{proof}

Now we can write 
\begin{equation}\label{e.se261}
\begin{aligned}
\Pb[\cU_{\eps}(t,x,\dot{W}_\delta)\leq r] \leq &\Pb[\tfrac{1}{2} e^{-\sqrt{\lambda} \mathrm{dist}(\dot{W}_\delta,A_\lambda(t,x))}\leq r]\\
\leq &\Pb\left[\mathrm{dist}(\dot{W}_\delta,A_\lambda(t,x)) \geq \tfrac{\log (2r)^{-1}}{\sqrt{\lambda}} \right],
\end{aligned}
\end{equation}
where $\mathrm{dist}(\dot{W}_\delta,A_\lambda(t,x))=\inf\{ \|\dot{W}_\delta-\dot{\W}_\delta\|_{L^2(\R^3)}: \dot{\W}_\delta\in A_\lambda(t,x)\}$. Now we can apply \cite[Lemma 4.5]{hu2018asymptotics} to derive that 
\begin{equation}\label{e.se262}
\Pb\left[\mathrm{dist}(\dot{W}_\delta,A_\lambda(t,x)) \geq \tau+2 \sqrt{\log\tfrac{2}{c}} \right] \leq  2e^{-\tau^2/4}
\end{equation}
for all $\tau>0$, where $\lambda,c>0$ are chosen as in Lemma~\ref{l.lowerbd} and are independent of $\eps,\delta>0$ and $t\in[0,T]$. Combining \eqref{e.se261} and \eqref{e.se262}, we have 
\[
\Pb[\cU_\eps(t,x,\dot{W}_\delta)\leq r] \leq  2\exp\left( -\tfrac14\left(\tfrac{\log (2r)}{\sqrt{\lambda}}+2\sqrt{\log \tfrac{2}{c}}\right)^2\right),
\]
which implies $\E[\cU_\eps(t,x,\dot{W}_\delta)^{-n}]\les1$ and completes the proof of \eqref{e.negmm}.


\end{document}